\documentclass[sn-mathphys]{sn-jnl}

\usepackage{graphicx}
\usepackage{xcolor}
\usepackage{ulem}
\usepackage{amsmath,amsthm}

\usepackage{amsmath}
\usepackage{amssymb,bm}
\usepackage{amsfonts}
\usepackage{graphicx}
\usepackage{afterpage}
\usepackage{flafter}
\usepackage{marvosym}
\usepackage{multirow}
\usepackage{dcolumn}
\usepackage{lipsum}
\usepackage{epstopdf}
\usepackage{algorithm}
\usepackage{color}

\usepackage{cleveref}
\newtheorem{thm}{Theorem}
\newtheorem{prop}[thm]{Proposition}

\newtheorem{cor}[thm]{Corollary}
\newtheorem{rem}[thm]{Remark}

\newcommand{\eps}{\varepsilon}
\newcommand{\N}{\ensuremath{\mathbb{N}}}

\begin{document}

\title[Kicked FitzHugh-Nagumo chains]{Periodically kicked feedforward chains of\\ simple excitable FitzHugh-Nagumo neurons }

\author*[1]{\fnm{Benjamin} \sur{Ambrosio}}\email{benjamin.ambrosio@univ-lehavre.fr}

\author[2]{\fnm{Stanislav~M.} \sur{Mintchev}}\email{stanislav.mintchev@cooper.edu}
\equalcont{These authors contributed equally to this work.}

\affil*[1]{\orgdiv{Department of Mathemathics}, \orgname{Normandie Univ, UNIHAVRE, LMAH,  FR CNRS 3335, ISCN}, \orgaddress{\street{25 rue Philippe Lebon}, \city{Le Havre}, \postcode{76600}, \state{Normandie}, \country{France}}}

\affil[2]{\orgdiv{Department of Mathematics}, \orgname{The Cooper Union for the Advancement of Science and Art}, \orgaddress{\street{41 Cooper Sq.}, \city{New York}, \postcode{10003}, \state{NY}, \country{USA}}}






\abstract{
This article communicates results on regular depolarization cascades in periodically-kicked feedforward chains of excitable two-dimensional FitzHugh-Nagumo systems driven by sufficiently strong excitatory forcing at the front node. The study documents a parameter exploration by way of changes to the forcing period, upon which the dynamics undergoes a transition from simple depolarization to more complex behavior, including the emergence of mixed-mode oscillations. Both rigorous studies and careful numerical observations are presented. In particular, we provide rigorous proofs for existence and stability of periodic traveling waves of depolarization, as well as the existence and propagation of a simple mixed-mode oscillation that features depolarization and refraction in alternating fashion. Detailed numerical investigation reveals a mechanism for the emergence of complex mixed-mode oscillations featuring a potentially high number of large amplitude voltage spikes interspersed by an occasional small amplitude reset that fails to cross threshold. Further careful numerical investigation provides insights into the propagation of this complex phenomenology in the downstream, where we see an effective filtration property of the network; the latter amounts to a successive reduction in the complexity of mixed-mode oscillations down the chain.}


\maketitle

\section{Introduction}
In mathematical neuroscience, the study of nerve depolarization goes back to the original Hodgkin-Huxley equations for electrical propagation in the squid giant axon~\cite{Hod-1952}. That model features a nonlinear reaction-diffusion system that accounts for the ionic fluxes across the nerve cell membrane as well as the cell axon's properties as an electrical cable. A variety of models have since been derived from the Hodgkin-Huxley equations. Among them, the reduction proposed by FitzHugh and Nagumo~\cite{Fit-1961,Nag-1962} combines the effects of several  auxiliary ion channels into a single recovery variable and thus leads to a simpler dynamical system that is more amenable to mathematical analysis; the resulting PDE has been analyzed extensively both theoretically and in applied context~\cite{Fen-2008,Has-1975,Has-1976,Mck-1970,Jon-1984,Kop-1991,Rau-1978,Rin-1983}.

Some recent studies have focused on a (space-)non-homogeneous reaction-diffusion version of the FitzHugh-Nagumo model~\cite{Amb-2017,Amb-2009}; this work was done with particular emphasis on the interplay between the oscillatory and excitatory properties of the traditional two-dimensional ODE system. This interplay provides a setting wherein localized oscillations (realized by tuning a specific part of the medium to the oscillatory regime) diffuse into the surrounding excitable medium. The studies examined a class of Neumann boundary condition problems for the PDE
\begin{align}\label{E:benjamin-prior_work}
\begin{cases}
\eps u_t = f(u) - v + d \, u_{xx},\\
v_t = u - c.
\end{cases}\tag{RD FHN}
\end{align}
A crucial feature of this model was the incorporation of a spacial non-homogeneity $c\,(=c(x))$ in the equation of the recovery variable $v$. This feature afforded a mechanism for the generation of a periodic signal in some small region of the spacial domain, which may in turn propagate or fail to do so depending on the excitability properties (which vary according to $c(x)$) of neighboring regions. From a mathematical perspective, the success or failure of the propagation of oscillations corresponds to an interesting bifurcation phenomenon in this infinite-dimensional setting.  Similar properties have been exploited in a more applied context~\cite{Maia-2013} where the authors have used a FitzHugh-Nagumo reaction-diffusion system with a non-homogeneous Laplacian term to model axon damage resulting from traumatic brain injury.

\begin{figure}[t!]
\begin{center}
\includegraphics[scale=0.4]{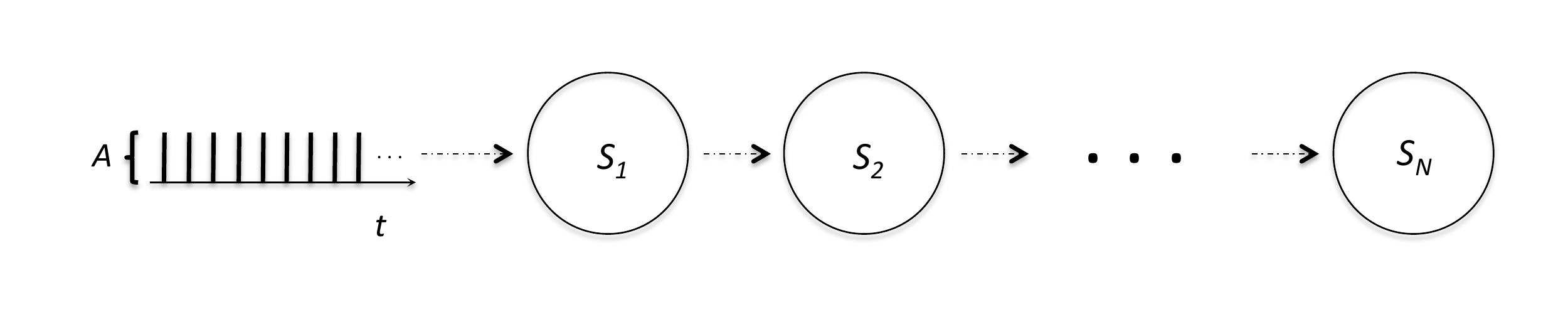}
\caption{Network diagram. A periodic train of Dirac stimuli is injected into $S_1$; for $j \geq 2$ unidirectional coupling provides communication between $S_{j-1}$ and $S_j$ via an instantaneous kick to the recovery variable of the receiving cell, when its left neighbor's voltage crosses the depolarization threshold.}
\label{fig:kicked_ff_FHN_chain-mod}
\end{center}
\end{figure}
The present article provides and analyzes a related simple model for wave propagation phenomena in spatially-discrete unidirectional chains of excitable nerve cells. Each cell in the chain has voltage dynamics modeled by a two-dimensional FitzHugh-Nagumo (FHN) ordinary differential equation. Continuous dependence in space is replaced by a simple coupling scheme through instantaneous kicks to the recovery variable. More precisely, with reference to \cref{fig:kicked_ff_FHN_chain-mod} we consider a feedforward chain of FitzHugh-Nagumo neurons $\{ S_j \}_{j \geq 1}$ modeled by
\begin{align}
&S_1: \qquad
\begin{cases}
\eps u_t = f(u) - v,\\ \notag
v_t = u - c - A \sum_{i=0}^{\infty} \delta(t-t_i);
\end{cases}\\ \tag{FHN FFC}\label{E:main_model}\\
&S_j: \qquad
\begin{cases}
\eps u_t = f(u) - v,\\ \notag
v_t = u - c - A \,  \delta(u_{j-1}-K) \cdot \mathbf{1}_{\{ v_{j-1} < 0 \}},
\end{cases} \qquad (j>1)
\end{align}
where $t_i=\alpha i$ for some fixed period $\alpha$, and $\eps$ is small but positive. The $t$-subscripts denote ordinary differentiation with respect to time. This chain is spatially homogeneous, in the sense that the function $f$ as well as the parameters $\eps$, $c$, $A$, and $K$ are all $j$-independent. In the second set of equations, we employ the convention $\infty \cdot 0 = 0$. As in the standard FHN model, the function $f$ is assumed to have a graph resembling the general shape of a negative-leading-coefficient cubic; for explicit calculations and numerical simulations, we assume $f(u)=3u-u^3$. The parameter $A$ represents the size of Dirac stimuli presented in periodic fashion to the chain's front node $S_1$, or else arriving at $S_j$ for $j>1$ when the voltage variable $u_{j-1}$ of the preceding cell $S_{j-1}$ crosses (on upswing) a certain threshold value $K$ (we also take $K=0$ for numerical simulations); in either case, these stimuli are incorporated into the dynamics through the recovery variable $v_j$.

A detailed investigation of \cref{E:main_model} is motivated in part by the necessity for a discrete-space alternative to~\cref{E:benjamin-prior_work} that features comparable phenomenology by way of significantly simpler dynamical mechanisms. The most natural analog of~\cref{E:benjamin-prior_work} studied in~\cite{Amb-2017,Amb-2009} involves a discretized Laplacian term in the right-hand side of the $u$ equation as well as a dependence of $c$ on the index / location of the cell within the network. This correspondence yields a coupled map lattice whose complexity is at least as high as that of the original partial differential equation (see~\cite{Chaz-2005}, esp. relevant chapters~\cite{Afra-2005,Bae-2005,Flo-2005}); accordingly, an analysis of the corresponding coupled map lattice would be just as much, if not more complicated than an analysis of the original reaction-diffusion model. While it is typical for neural coupling to enter the cell dynamics directly through the voltage variable~\cite{Erm-2010, Izh-2007}, \cref{E:main_model} features the well-studied dynamics of the isolated FitzHugh-Nagumo system as perturbed only occasionally by instantaneous kicks delivered to the recovery variable. The ensuing dynamical system is readily amenable to a mathematical treatment -- we will discuss both rigorous results and detailed numerical work, all of which point to a variety of interesting and biophysically relevant bifurcation phenomena that can be established analytically. This simplification notwithstanding, the choice of coupling for our model is still appropriate since it presents a second-order cumulative effect to the voltage while serving to emphasize the threshold barrier that is a common feature of depolarization in the presence of excitability. (Specifically, for the example of the cubic nullcline featured in the FitzHugh-Nagumo system, the recovery variable's clearance of the local minimal point effectively induces a rapid jump in the voltage variable -- this serves precisely the purpose of guaranteeing a mechanism for the generation of action potentials via excitation.) We note also that this simple coupling formalism is related to common modeling paradigms of collective network behavior in mathematical neuroscience. For instance, forcing the network's front node by an external input current, together with Dirac coupling between cells through the recovery variable is commensurate with, e.g., the usual integrate-and-fire coupling scheme where communication between cells is achieved by discrete kicks/jumps in the synaptic current variable prior to being incorporated into the voltage equation. Thus, the proposed model~\cref{E:main_model} features a simple coupling scheme analogous to those commonly employed in the formulation of spiking networks~\cite{Char-2016,Char-2015,Ger-2002}.

\subsection{Further notations, definitions, and basic properties of FHN}
\label{subsec:basic_props_FHN}

The isolated FHN model $S_{\bullet}$ with $A\equiv 0$ features a rest state at $(c,f(c))$ and is an example of a {\it slow-fast} system. The $u$ nullcline $v=f(u)$ is also the system's {\it critical manifold}, wherein trajectories are constrained for $\eps = 0$. For $0 < \eps \ll 1$, away from the critical points for the graph of $f$, singular perturbation theory guarantees the existence of a flow-invariant curve $\eps$-close to this nullcline, where trajectories follow the slow timescale; such a curve is referred to as a {\it slow manifold} for the system. In this regime, we have that $u_t$ is positive below the $u$ nullcline and negative above it; accordingly, the flow above a slow manifold (say, above and sufficiently separated from $v=f(u)$) is expected to move to the left (negative $u$ direction) according to the fast timescale, and one has precisely the oppositely-directed fast motion below.

We focus on the parameter regime $c<-1$, for which it is well-known that the system is {\it excitable}~\cite{Izh-2007,Kee-2009}. Starting from a certain point $(c,v_0)$, with $v_0$ sufficiently smaller than $f(c)$, one can clearly construct a conjunction of four singular orbits consisting of 

\medskip
\begin{minipage}[t]{0.7\textwidth}
\begin{enumerate}
\item[$\sigma_1$:] a horizontal (fast) trajectory to the right from the point in question to the point $(u_0,v_0)$ satisfying $f(u_0)=v_0$; 
\item[$\sigma_2$:] a critical trajectory within the curve $v=f(u)$ from $(u_0,v_0)$ to the right critical point of $f$, say $(u_c,v_c)$;
\item[$\sigma_3$:] a horizontal (fast) trajectory to the left from the right critical point of $f$ to the other point in $v=f(u)$ satisfying $v=v_c$;
\item[$\sigma_4$:] a critical trajectory within the curve $v=f(u)$ connecting the last point $(u_l,v_c)$ to the equilibrium $(c,f(c))$.
\end{enumerate}
\end{minipage}

\vspace{0.5cm}

Accordingly, for $\eps \approx 0$, singular perturbation theory~\cite{Feni-1979,Jon-1995,Kru-2001} assures the existence of a orbit $\gamma_{\bullet}(t)$ of the isolated FHN starting from $(c,v_0)$ that is asymptotic to $(c,f(c))$ and $\eps$-close to the concatenation $\sigma_1 \cup \sigma_2 \cup \sigma_3 \cup \sigma_4$ described above. For the purposes of this article, excitability is to be interpreted in the following sense: the system is at rest at $(c,f(c))$, but if presented with a sufficiently large instantaneous momentum transfer in the negative $v$ direction, the system crosses threshold to the right (this constitutes an effective {\it depolarization}), and so resets to the rest state only after this threshold crossing; alternatively, if the kick is not sufficiently large, reset to the rest state happens directly without a depolarization. Strong excitation with discharge is thus possible, but not necessarily achieved.

Of note, the phase portrait for an isolated FHN cell contains a rich geometry that provides a mechanism for {\it mixed-mode oscillations} (MMOs). Recall that MMOs correspond to switching between small amplitude and large amplitude oscillations. 
Such patterns were first discovered in the Belousov-Zhabotinsky reaction~\cite{Bro-1991,Rot-2003} and have since been frequently observed in experiments and models of chemical and biological rhythms, particularly in the context of neuroscience~\cite{Izh-2007,Rot-2008,Gha-2017}. There is a large number of works dealing with MMOs -- both the theoretical and the applied aspects. We refer to \cite{Des-2012} for a general review of various results on MMOs as well as \cite{Wec-2007} for a short summary. A remarkable way to obtain MMOs is through the canard phenomenon found in slow-fast systems. Canard solutions are phase space trajectories that follow a repulsive manifold for a certain amount of time. They were first discovered in the Van der Pol oscillator~\cite{Ben-1981} with non-standard analysis techniques, and this led to numerous subsequent works, for example~\cite{Amb-2018,Dum-1996,Eck-1983,Kru-2001}. In this context, small oscillations may correspond to canards that are repelled to a nearby attractive manifold, while large oscillations arise from canards which are repelled to a distant one. Combined with a return mechanism, this provides a setting for sustained MMOs~\cite{Bro-2006,Des-2012,Kru-2014}. 

The model we investigate also features the emergence of MMOs within an individual cell as well as their propagation through the feedforward chain. To draw a distinction between our case and the examples from the literature referenced above, we note that the first of these phenomena is based on an interplay between the impulsive periodic forcing and a phase space geometry that allows both relaxation oscillations and simple reset to equilibrium. Variations to the period of the forcing can lead to multiple consecutive relaxation oscillations (see \cref{fig:u-v-differentcolors}) that are occasionally followed by a reset to equilibrium. Some of the observed trajectories in this model are related to the canard phenomenon since they seem to follow the repulsive manifold for a non-trivial amount of time. Overall the dynamical mode is characterized by the property that a solution curve 
consists of multiple loops about the equilibrium $(c,f(c))$, some of which are sufficiently large to cross the depolarization threshold at $u=0$, while others remain too small for this. This amounts to an evolution of the voltage variable $u$ that features a combination of small-amplitude flutters and large-amplitude excursions, the latter resulting in the triggering of synaptic firing toward the neighboring cell. Since this complex phenomenology provides a mechanism for skipping some of the beats in a regular depolarization train, it also establishes a way for cells in this network to support a richer spike train encoding structure. We discuss this regime in some detail in \cref{sec:mmo_at_S_1}, and we explore the downstream propagation of MMOs in \cref{sec:propagation_MMO}.

We denote by $\underline{\varphi} (\cdot) = \left( \varphi_j (\cdot) \right)_{j = 1}^N$ the unique solution to the version of \cref{E:main_model} featuring $N$~cells evolving dynamically from a prescribed initial value $\underline{\varphi} (0) = (\varphi_{\cdot}(0)) \in (\mathbb{R}^2)^N$. We will be particularly interested in the initial prescription $\varphi_j(0)=(c,f(c))$ for all $j$, corresponding to the scenario where the chain is initially kicked from its rest state. In accordance with standard terminology, a solution $\underline{\varphi}$ to such an initial value problem is said to be {\it periodic of period $\tau$} if
\[
\underline{\varphi}(t + \tau) = \underline{\varphi}(t),
\]
for all $t \geq 0$. We refer to $\underline{\varphi}$ as a {\it traveling wave solution (TW)} provided that there exist a function $\psi : [0,\infty) \to \\mathbb{R}^2$ and a $j$-independent time shift $\beta \in \mathbb{R}$, such that
\[
(\underline{\varphi}(t))_j = \varphi_j (t) = \psi (t - j \,  \beta), \qquad j \in {1,\ldots,N},
\]
for all $t \geq 0$. We note that, in the context of the treatment presented herein, a periodic TW in \cref{E:main_model} corresponds to the presence of rhythmic and regular cascades of depolarization in the chain. Related phenomenology has been observed in the context of type~I phase oscillator networks~\cite{Fer-2016,Lan-2014} under excitatory coupling, where a variety of rigorous and numerical results point to the existence and stability (both Lyapunov and structural) of traveling wave solutions in homogeneous feedforward chains. In addition to its connections with simple spiking network models, our model provides an excellent opportunity to address questions concerning wave propagation in the context of excitable cells, which differ 
from oscillatory neurons because of the attraction to a resting equilibrium in the absence of stimulus. Studying side-by-side oscillatory and excitatory models may reveal the extent to which the presence of stable regular waves of activity is a stimulus-driven phenomenon. 
If we suppose that cell $S_1$ receives a kick (by $-A$ in the $v$-direction) sufficiently large to induce a simple fast trajectory across threshold resulting in depolarization, the second cell $S_2$ in the chain is also affected by the stimulus. Thus results a simple setting for exploring depolarization propagation in chains of excitable units coupled according to the equations above. This relates to the work \cite{Amb-2017,Amb-2009} where traveling pulses are generated from a central generator and diffuse outwards, according to the FitzHugh-Nagumo reaction-diffusion model~\cref{E:benjamin-prior_work}. As in that prior work, of interest to the present study -- will periodic excitation at $S_1$ propagate down the chain or be blocked? If it propagates, will that propagation be akin to a periodic wave with a well-defined velocity (like an effective diffusion in a homogeneous medium), or can the propagation be more complex? 

\subsection{Plan for the paper}
\label{subsec:plan}

The paper is organized as follows. In \cref{subsec:main_results} we give an overview of our results, with emphasis to distinguish rigorous work from deductions based solely on thorough numerical simulations; these results are summarized in \cref{fig:result_diagram}. In \cref{sec:simple_TW}, we establish a theoretical result guaranteeing the presence of a Lyapunov-stable traveling wave solution with periodicity that matches the forcing period. In \cref{sec:mmo_at_S_1} we provide a rigorous result establishing a mechanism for especially simple mixed-mode oscillations at $S_1$ under periodic forcing with kicks. We also present numerical investigations of the variety of possibilities for more complex mixed-mode oscillations at $S_1$. Then, in \cref{subsec:propagation_simple_mixed}, we relate our investigations of mixed-mode oscillation phenomenology at $S_1$ to downstream propagation, which can be significantly more complicated than the simple scenario established in \cref{sec:simple_TW}. We conclude with a survey of numerical experiments on propagation when dynamics at one or more of the upstream sites is complex and poorly understood; these numerics are described in \cref{subsec:propagation_complex_mixed-numerics} and seem to indicate that complex behavior may be successively averaged / filtered out by the first few nodes of a chain, so as to obtain tame behavior downstream -- such as a traveling wave -- while the first few upstream sites behave in a more complicated manner. \Cref{sec:conclusion} provides some concluding remarks as well as directions for future work.

\section{Main results} \Cref{tab:parameter_values-numerics} lists the model parameter values used for the numerical simulations of~\cref{E:main_model}. The reader should refer to \cref{fig:result_diagram} for a summary of the results detailed below.
\label{subsec:main_results}

\begin{table}[t]
{\footnotesize
\caption{Values for the parameters used in the numerical simulations}\label{tab:parameter_values-numerics}
\begin{center}
\begin{tabular}{|c|r|} \hline
\bf Parameter & \bf Value \\ \hline
$\eps$ & $0.1\phantom{00}$ \\
$c$ & $-1.2\phantom{00}$ \\
$A$ & $1.0\phantom{00}$ \\
$K$ & $0\phantom{.000}$ \\
RK4 timestep & $0.001$ \\ \hline
\end{tabular}

\end{center}
}
\end{table}

\subsection{Low frequencies (rigorous)}
If the frequency of incoming kicks is sufficiently low, $S_1$ is guaranteed to settle to within a small neighborhood of $(c,f(c))$ prior to any subsequent stimulation. For the purposes of this discussion, we treat this as having the forcing period $\alpha$ sufficiently large (say, there is some $\alpha_0>0$ and the preceding remark holds for $\alpha > \alpha_0$). Since we assume the kick size $A$ to be sufficiently large for the vertical offset induced by a kick to clear the neighborhood of the critical manifold where the system can feature complex behavior (specifically, we have in mind avoiding canard trajectories; see, e.g.,~\cite{Dum-1996}), the resulting dynamics at $S_1$ will be time asymptotic to a periodic threshold traversal with period $\alpha$, which guarantees $S_2$ receives a periodic sequence of kicks identical to that presented to $S_1$.

The argument outlined above clearly extends inductively because of the homogeneity in the network. The result is an asymptotic trend toward a solution close to an {\bf $\alpha$-periodic traveling wave (TW)} down the chain, a phenomenon resembling that studied for phase oscillator ensembles in~\cite{Lan-2014} and~\cite{Fer-2016}. Of note, there are some theoretical subtleties to be detailed during proofs.  Nevertheless, the network supports traveling waves of arbitrary (sufficiently high) period, when the stimulus / kick size is sufficiently large in the sense presented above. We establish this rigorously in \cref{sec:simple_TW}.

\begin{figure}
\begin{center}
\includegraphics[scale=0.45]{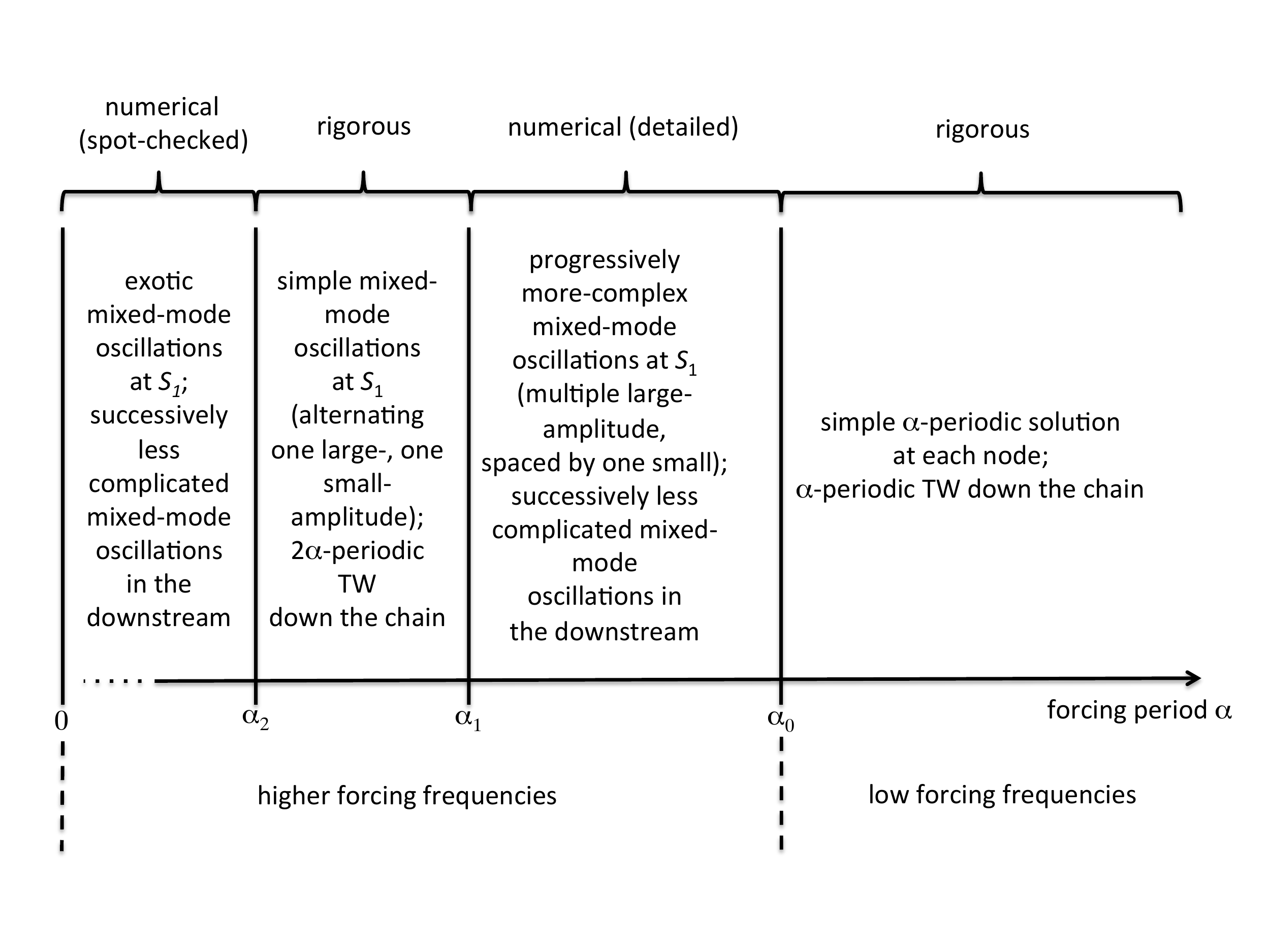}
\caption{A schematic of the results discussed in this paper. Rigorous analysis at $S_1$ for forcing with $\alpha$-periodic Dirac kicks reveals $\alpha$-periodic response (at low frequency, with $\alpha > \alpha_0$) and $2\alpha$-periodic simple mixed-mode oscillation (at higher frequency, with $\alpha \in (\alpha_2,\alpha_1)$); in turn, this analysis extends inductively down the chain, thus establishing the existence of  $\alpha$-periodic TW at low frequency (resp., $2\alpha$-periodic TW for an interval of higher frequencies). A detailed numerical investigation explores a bifurcation cascade as the period $\alpha$ is increased across the threshold $\alpha_1$ and then dialed up to the threshold $\alpha_0$; this analysis reveals successively more complex mixed-mode oscillations at $S_1$, whose complexity decreases as they propagate downstream -- that is, the chain filters out the complexity and decreases the number of large amplitudes before a small amplitude break (when examining the dynamics at $S_j$, as $j$ increases). A few numerical case studies of this filtering phenomenon (at $S_j$, with $j$ low) are also presented for illustration in the very high forcing frequency domain, where $\alpha < \alpha_2$.}
\label{fig:result_diagram}
\end{center}
\end{figure}

\subsection{Higher frequencies (analysis and numerics)}
As the forcing period $\alpha$ is decreased, a complicated interplay between stimulus size and frequency begins to emerge. This interplay provides a setting for very diverse phenomenology that includes a variety of {\bf mixed-mode oscillations (MMOs)} as well as obstruction to depolarization. Thinking of the stimulus/kick size as fixed going forward, exploring the interplay amounts to a parameter study by way of decreasing $\alpha$.\footnote{Of interest for the feedforward chain network studied here, how do MMOs propagate downstream (do they regularize or become more irregular?), and are there regimes in which wave propagation / signal transmission downstream is altogether blocked?} See \cref{fig:u-v-al=8p3,fig:u-v-al=8p4,fig:u-t-al=8p4,fig:u-v-t-al=8p41,fig:u-t-al=8p41,fig:u-v-t-al=8p45,fig:u-t-al=8p45}. 

Setting aside the observation that a full half-open interval $(\alpha_0,\infty)$ of forcing periods features the tame behavior discussed above, we continue the $\alpha$-parameter exploration. For moderately higher frequency, kicks are overwhelmingly likely to arrive while the dynamics is in a slow regime (and yet, before it reaches a small neighborhood of the equilibrium). There are two possibilities -- a kick arrives either on the downswing toward the stable equilibrium $(c,f(c))$, or on the upswing slow branch to the right. Since the kicks always decrease the value of the recovery variable $v$, the former possibility can yield a fast subsequent trajectory that either crosses threshold for depolarization (provided that kick arrives sufficiently late) or wavers for a small amplitude oscillation around $(c,f(c))$, close to the left extremum of the critical manifold (see \cref{fig:u-v-al=8p3,fig:u-v-al=8p4,fig:u-v-t-al=8p41,fig:u-v-t-al=8p45}); on the other hand the latter possibility yields various interesting complications, including a regime where the neuron is trapped in small amplitude oscillations on the right (upswing) slow branch and never resets to the left slow branch, which drives the voltage $u$ to $+\infty$ and prevents it from undergoing future depolarizations.

At some point, provided that $\alpha$ is within an appropriate range, one simple mechanism for MMO at $S_1$ emerges as follows. For simplicity, assume the cell starts at its natural equilibrium $(c,f(c))$. The kick received to start is sufficiently large to induce a fast horizontal trajectory (assume $f(c)-A$ is sufficiently smaller than $f(-1)=-2$) across threshold; thereafter, the subsequent stimulus received at $t=\alpha$ comes when the trajectory is in the refractory phase, but not yet sufficiently close to the equilibrium to guarantee clearing the threshold anew. As a consequence, the subsequent stimulus induces a much smaller oscillation consisting of a direct attraction to the equilibrium instead of a large oscillation. Thereafter, over the next $\alpha$ interval, the cell has sufficient time to reach the aforementioned neighborhood of the equilibrium that is essential to assuring that it clears the firing threshold at the subsequent stimulation; see \cref{fig:u-v-al=8,fig:u-t-al=8,fig:v-t-al=8}. In \cref{sec:mmo_at_S_1} we prove that in this regime, the dynamics of $S_1$ features a  $2\alpha$-periodic solution with this geometry. For definiteness, we define the range of $\alpha$ where these simple MMOs are supported as $(\alpha_2,\alpha_1)$. Accordingly, when $\alpha$ is between $\alpha_1$ and $\alpha_0$, more complex mixed-mode oscillations are present; see~\cref{sec:mmo_at_S_1} for details. On the other hand, the MMO property is absent when $\alpha$ is only slightly below $\alpha_2$.\footnote{See our discussion of the state of affairs for $\alpha$ only slightly smaller than $\alpha_2$ in~\cref{rem-alpha_less_than_alpha2}.}

\begin{table}
{\footnotesize
\caption{Numerically-observed (roughly approximate) critical values for the forcing period (see \cref{fig:result_diagram}), measured in the integral units of the variable $t$; this data is obtained with model parameters fixed according to \cref{tab:parameter_values-numerics}.}\label{tab:observed-critical_period_values}
\begin{center}
\begin{tabular}{|c|c|} \hline
\bf Label & \bf Value \\ \hline
$\alpha_0$ & $8.5$ \\
$\alpha_1$ & $8.2$ \\
$\alpha_2$ & $7.5$ \\ \hline
\end{tabular}
\end{center}
}
\end{table}
The simple MMO at $S_1$ implies that the stimulus presented to $S_2$ amounts to a periodic sequence of kicks with period $2\alpha$. Of course, if $2\alpha > \alpha_0$, the overall effect is attraction to a solution close to a $2\alpha$-periodic {\bf  traveling wave, with period equal to twice the forcing period} for all $S_j$ with $j \geq 2$ (again, this will be shown rigorously) -- and this is the first instance in which we have a complete understanding of phenomenology beyond the simplest traveling wave for low frequency. But it is also possible that $2\alpha < \alpha_0$; and in the latter case, we note there may yet be a $4\alpha$-periodic oscillation of simple mixed-mode type at $S_2$, or a vastly more complicated situation; in the former case, the prior reasoning extends inductively; in the latter, we illustrate typically observed behavior through a few representative numerical simulations. In particular, these latter simulations speak to the manner in which MMOs are transmitted down the feedforward chain. Our numerics indicate that when they arise, complex MMOs induce less complex behavior (sometimes of MMO type, not always) down the chain; we investigate this attenuation further in~\cref{subsec:propagation_complex_mixed-numerics}.

\section{Existence of simple $\alpha$-periodic TW}
\label{sec:simple_TW}

We begin with a detailed study of the behavior of $S_1$ when $A>0$, assuming that the system begins at initial condition $(c,f(c))$. If it were to receive a single stimulus, this neuron would instantaneously find itself at $(c,v_0)$, with $v_0=f(c)-A$. Clearly, if $A$ is sufficiently large, the system will feature a large trajectory orbit (including a depolarization event) that is asymptotic to $(c,f(c))$. Since $\epsilon$ is assumed to be small, our proofs are often inspired by the case $\epsilon=0$. In this case, the trajectory after kick is no more continuous, but evolves first on the right part of the critical manifold, and jumps to the left part when $u$ reaches the value $1$.   There are however some differences between the case $\epsilon=0$ and $\epsilon > 0$. For example, if $\epsilon=0$ one can prove existence and uniqueness of the periodic solution at $S_1$. But the argument of uniqueness fails for $\epsilon>0$. Therefore, we consider below the  cases $\epsilon=0$ and $\epsilon>0$ separately.  

\subsection*{\textbf{The limit case $\epsilon=0$}}

\begin{prop}\label{thm-simple_periodic-eps=0}
 We assume $\epsilon=0$ and  $A>0$ fixed sufficiently large. If $\alpha$ is sufficiently large,  $S_1$ admits a unique  $\alpha$-periodic trajectory which crosses the firing threshold $u=K$. Furthermore, this trajectory is stable.
\end{prop}

\begin{figure}
\begin{center}
\includegraphics[width=7cm, height=5cm]{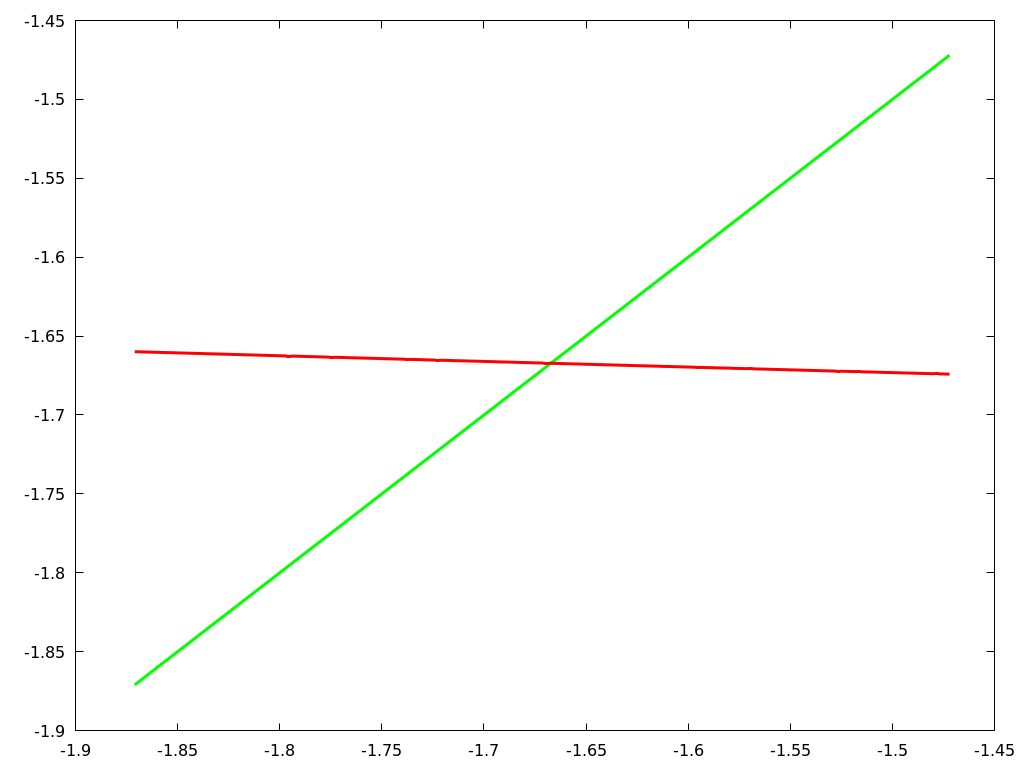}
\includegraphics[width=7cm, height=5cm]{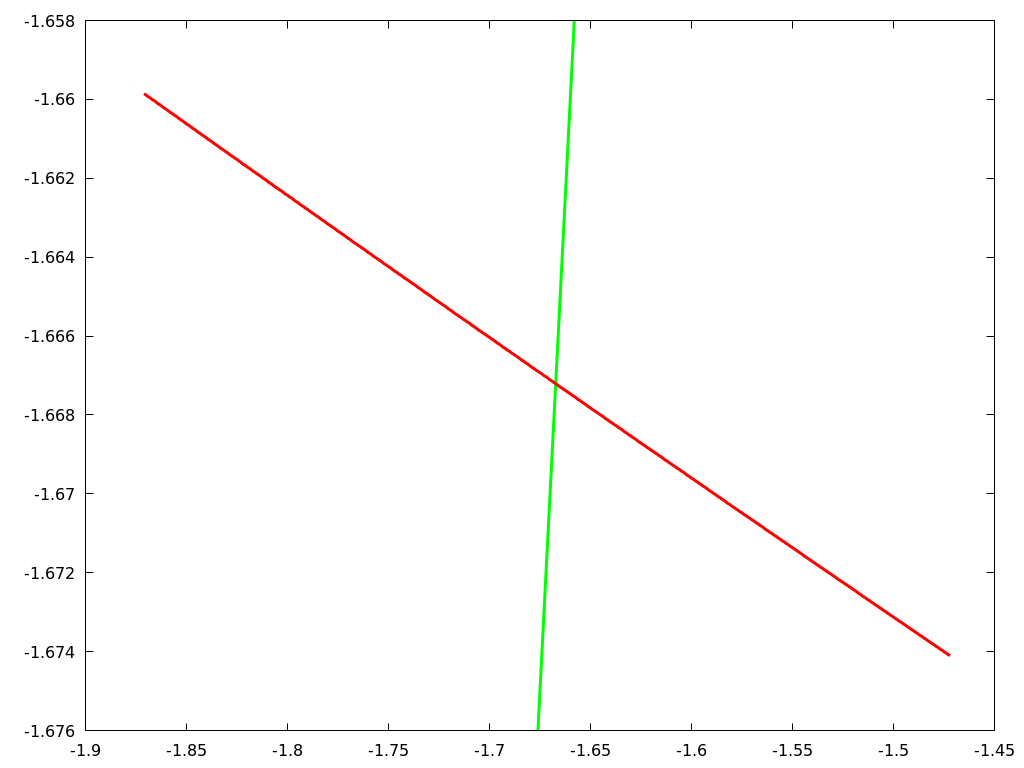}
\caption{The map $F$ (in red) and its fixed point (in green, the identity), for $\eps=0$ and $\alpha=10$. This picture illustrates the local stability of the periodic solution.}
\label{fig:F}
\end{center}
\end{figure}

\begin{figure}
\begin{center}
\includegraphics[width=7cm, height=5cm]{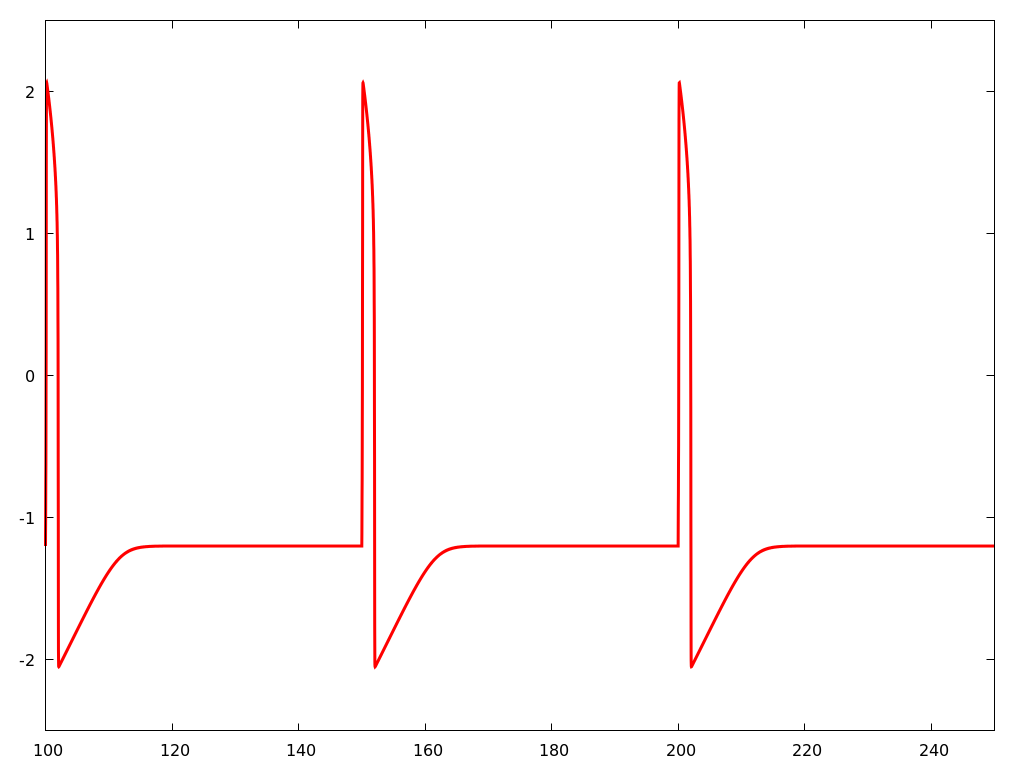}
\caption{Steady state evolution of  $u_1$ as a function of time. $\alpha=50$.}
\label{fig:u-t-al=50}
\end{center}
\end{figure}

\begin{figure}
\begin{center}
\includegraphics[width=7cm, height=5cm]{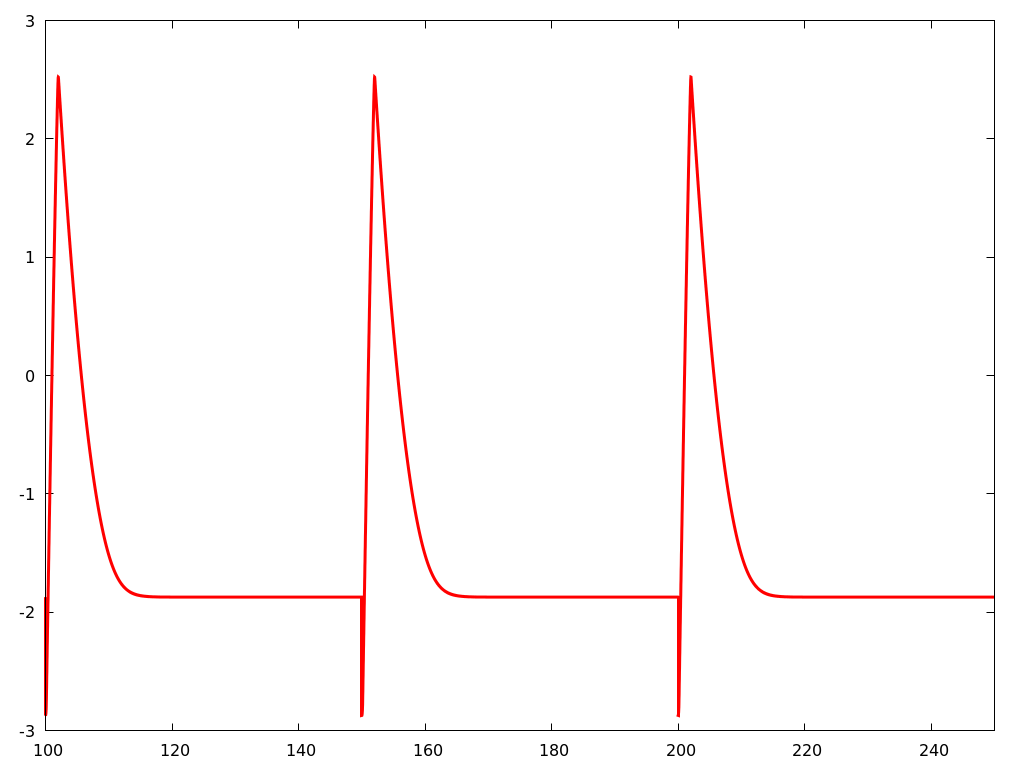}
\caption{Steady state evolution of  $v_1$ as a function of time. $\alpha=50$.}
\label{fig:v-t-al=50}
\end{center}
\end{figure}

\begin{figure}
\begin{center}
\includegraphics[width=7cm, height=5cm]{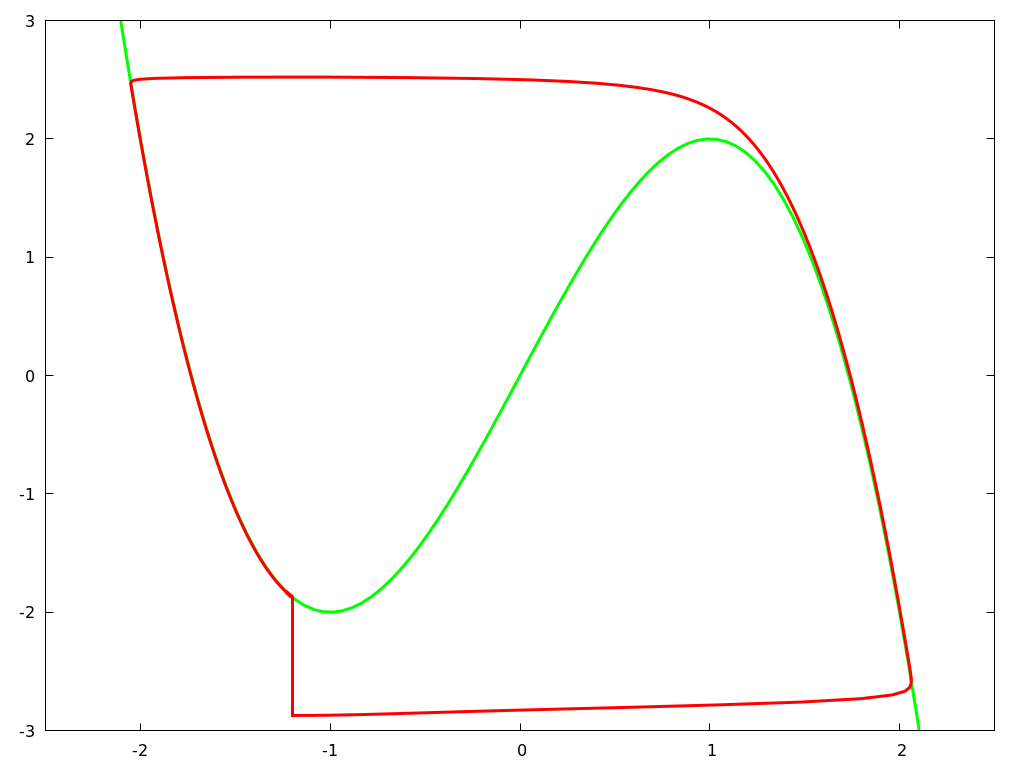}
\caption{Steady state evolution of  $\varphi_1=(u_1,v_1)$ as a function of time (red) and the graph of $v=-u^3+3u$ (green). $\alpha=50$. Of note, here the value of $\alpha$ is well above the critical value $\alpha_0$. (See \cref{tab:observed-critical_period_values} and \cref{fig:result_diagram} for reference.)}
\label{fig:u-v-al=50}
\end{center}
\end{figure}

\begin{figure}
 \begin{center}
\includegraphics[width=7cm, height=5cm]{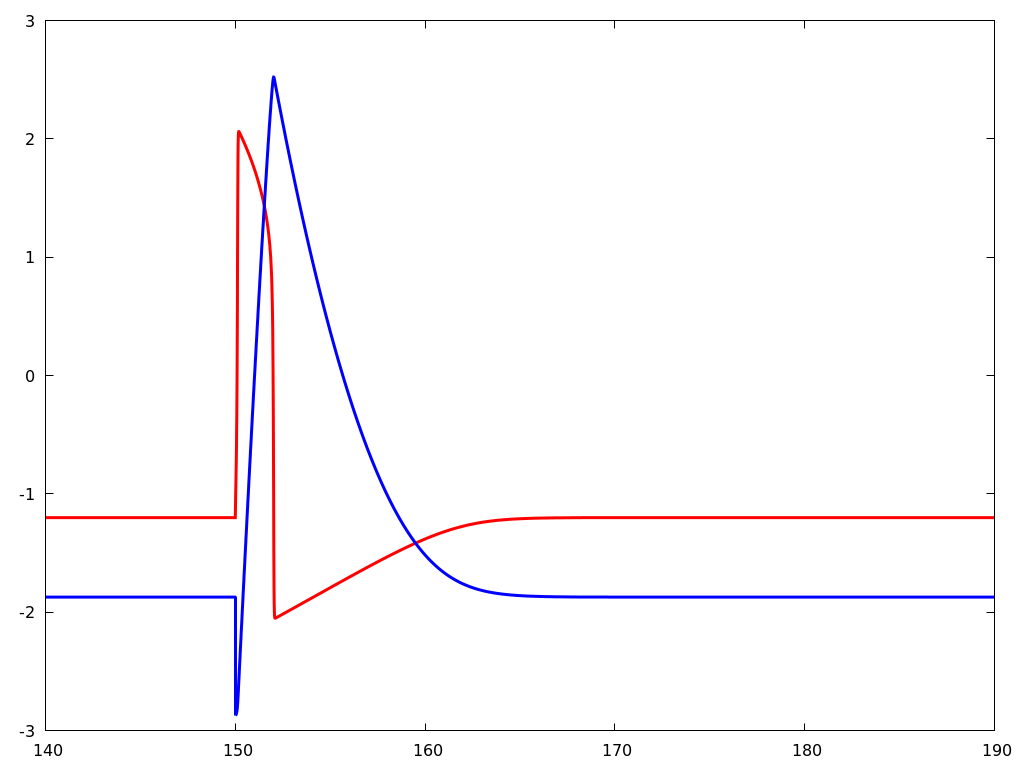}
\caption{Steady state plot of $u_1$ (red) and $v_1$ (blue) versus time; one period of oscillation is shown. $\alpha=50$.}
\label{fig:u-v-t-al=50-onep}
\end{center}
\end{figure}

\begin{proof}
\textbf{Existence} Let $F$ denote the discrete-time map that takes an initial condition on the left branch of the critical manifold $v=f(u)$ to its image under kick + time-$\alpha$ flow.  Notice that the curve under discussion can be parametrized by the $v$-coordinates of its points (1-1 correspondence). To prove existence of a periodic solution, it is sufficient to prove that $F$ admits a fixed point. First note that $F(f(c))>f(c)$. This results from the fact that $(c,f(c)$ is a stable stationary solution of the ODE. Now, we search for a point along the left branch of the critical manifold (above $(c,f(c))$) having the property that the trajectory ensuing from it returns with a smaller $v$-coordinate at time $t=\alpha$. Such a point  exists for sufficiently large $\alpha$ due to the attraction toward $(c,f(c))$. We denote $\bar{v}$ its $v$-coordinate. Then 
\[F(\bar{v})-\bar{v}<0 \mbox{ and } F(f(c))-f(c)>0 \]
By continuity of $F(v)-v$, there exists a value $v^*$ such that 

\[F(v^*)=v^* \]
which corresponds to a periodic solution.\\
\textbf{Uniqueness} 
We still denote $v^*$ the $v$-coordinate of the kick-point on the periodic solution exhibited above. Uniqueness follows from the observation that all trajectories ensuing from points with $v$-coordinates between $f(c)$ and the $v$-coordinate $v^*$  are mapped (after kick+time-$\alpha$ flow) to points with $v$-coordinates above $v^*$.  This results from the fact that any trajectory ensued below $v^*$ will spend more time on the right branch of the critical manifold than the trajectory ensued from $v^*$. A symmetric argument can be applied from above, and together these observations ensure that the periodic solution is the only trajectory whose pre-kick $v$-coordinate remains invariant under the discrete dynamics.\\
\textbf{Local Stability} 
We denote by $(u_1,v_1)(t)$ the $\alpha$-periodic solution. Let $(u_2,v_2)(t)$ be another solution with initial conditions satisfying $u_2(0)>u_1(0)$, $u_2(0)-u_1(0)$ small. And let

\medskip
\begin{minipage}[t]{0.7\textwidth}
\begin{itemize}
\item[-] $u_i(0^+), i \in\{1,2\}$ denote the values of $u_i$ on the right part of the critical manifold right after the kick;

    \item[-] $t_1$ denote the time at which  $(u_1,v_1)(t)$ reaches the fold point $(1,f(1))=(1,2)$; 
    
    \item[-] $\tau$ denote the time needed for $u_2$ to go from $u_2(0^+)$ to $u_1(0^+)$.
\end{itemize}
\end{minipage}

\vspace{0.5cm}
\noindent As a consequence,
\[
\begin{array}{rcl}
u_2(\alpha) &=&-2+\int_{t_1+\tau}^{\alpha}u'_{2}(s)ds\\ 
    &=&-2+\int_{t_1}^{\alpha-\tau}u'_{2}(s+\tau)ds\\ 
     &=&-2+\int_{t_1}^{\alpha-\tau}u'_{1}(s)ds\\ 
     &=&-2+\int_{t_1}^{\alpha}u'_{1}(s)ds-\int_{\alpha-\tau}^{\alpha}u'_{1}(s)ds\\
     &=&u_1(\alpha)-\int_{\alpha-\tau}^{\alpha}u'_{1}(s)ds\\
\end{array}    
\]

which leads to
\begin{equation*}
\begin{array}{rcl}
   u_1(\alpha)- u_2(\alpha) &=&\int_{\alpha-\tau}^{\alpha}u'_{1}(s)ds\\ 
   &<&\int_{0}^{\tau}(-u'_{2}(s))ds\\
   &<&-u_2(\tau)+u_2(0^+)\\
   &<&-u_1(0^+)+u_2(0^+)\\
   &<&-u_1(0)+u_2(0)
   
\end{array}    
\end{equation*}

since for $\alpha$ large enough $u_1(\alpha)$ is close of $c$ and therefore
\[\sup_{u\in[u_1(\alpha-\tau),u_1(\alpha)]}\frac{u-c}{f'(u)}<\inf_{u\in[u_1(0^+),u_2(0^+)]}(-\frac{u-c}{f'(u)}) \]
and also
\[-u_1(0^+)+u_2(0^+)<-u_1(0)+u_2(0)\]
since
\[\sup_{u\in[u_1(\alpha-\tau),u_1(\alpha)]} \lvert f'(u) \rvert <\sup_{u\in[u_1(0^+),u_2(0^+)]}f'(u) \]
\end{proof}
\begin{rem}
 \Cref{fig:F} gives a representation of $F$ for $\alpha=10$. It shows numerical evidence that 
\[ \lvert F'(v^*) \rvert <1 \]
which implies local stability. Note that $F$ can be computed almost explicitly. The calculation entails computing the times spent on the cubic -- first in the right branch, then in the left one after the jump -- and retrieving the value of $v$ at time $\alpha$; this leads to solving the equation
\[\int_{f^{-1}(v-A)}^1\frac{u-c}{f'(u)}du+\int_{f^{-1}(2)}^{f^{-1}(F(v))}\frac{u-c}{f'(u)}du=\alpha,\]
where $f^{-1}$ relates either to the right or left branch of the cubic.
\end{rem}

\subsection*{\textbf{The case $\epsilon>0$ (small)}}

We now proceed to the case $\eps>0$. The proof of the existence of a periodic solution is analogous to the case $\epsilon=0$. The proof of uniqueness does not extend to the case $\epsilon>0$. However, we show local stability, which in turn implies that the periodic solution is isolated. Stability is obtained by specific computations.
Let $F$ denote the mapping which maps an initial condition $(u,v)(0))$ to  $(u,v)(\alpha)$ corresponding to the point on the trajectory after kick + time $\alpha$ under the flow.
\begin{prop}\label{thm-simple_periodic-eps>0}
We assume  $A>0$ fixed sufficiently large. There exists $\epsilon_0>0$ and $\alpha_0>0$ such that for every $\epsilon \in (0,\epsilon_0)$ and for all $\alpha>\alpha_0$, $S_1$ has an  $\alpha$-periodic trajectory which crosses the firing threshold $u=K$, $i.e.$ there exists an initial condition $(u^*,v^*) \in \mathbb{R}^2$ such that the solution $(u,v)(t)$ ensued from there satisfies $(u,v)(\alpha)=(u^*,v^*)$.   Furthermore this trajectory is locally asymptotically stable, in the sense that there exists a neighborhood of  $(u^*,v^*)$  such that for every initial condition in this neighborhood 
\[\lim_{n\rightarrow +\infty}  (u,v)(n\alpha)=(u^*,v^*).\]
\end{prop}
\begin{proof}
The existence of a periodic solution follows from the Brouwer theorem. Consider a small enough closed ball $B$ around $(c,f(c))$; for $\alpha$ large enough, global asymptotic stability of $(c,f(c))$ for the ODE, ensures $F(B)\subset B$.\\
We now consider the stability.
We denote $(u_1,v_1)(t)$ the $\alpha$-periodic solution. Let $(u_2,v_2)(t)$ be another solution with initial conditions close to $(u_1,v_1)(0)$ .
Let 
\[g(t)=\epsilon(u_2-u_1)^2+(v_2-v_1)^2.\]
We have
\[g'(t)=2u^2\big(f'(u_1)-3u_1u-u^2 \big)\]
with
\[u=u_2-u_1.\]
It follows that 
\[g(\alpha)-g(0)=2\int_0^\alpha u^2\big(f'(u_1)-3u_1u-u^2 \big)dt.\]
Now, we remark that $u$ is controlled by the initial conditions, and the trajectories $u_1$ and $u_2$ are well known. Furthermore, a simple computation shows that the time for which  $f'(u_1(s))>0$ is of $O(\epsilon)$, while the time for which $f'(u_1(s))<0$ is uniformly greater than a positive constant. Altogether, assuming $u(0)$ small enough leads to the conclusion that
\[\int_0^\alpha u^2\big(f'(u_1)-3u_1u-u^2 \big)dt<0,\]
which means that
\[g(\alpha)<g(0)\]
and implies the result.
\end{proof}
\begin{rem}
The reader might think to apply the classical results for the stability of limit-cycles as in  \cite{Per-2001} p 216 (or \cite{And-1971} for the result with proof). However, looking into the details shows that the jump brings some technical difficulties with respect to the local change of variables around the periodic solution, which prevent a direct application of these results. The general condition  ensuring stability would write 
\[\int_0^\alpha f'(u_1(s))ds<0\]
which is finally very close to the argument provided here.
\end{rem}

\begin{rem}
As mentioned above,  the proof for uniqueness of the periodic solution provided for $\epsilon=0$ does not extend to the case $\eps>0$. The point is that when $\epsilon>0$ there is some time spent in the fast trajectories. A trajectory starting below another goes faster along the fast fibers but spend more time in the slow manifolds. It follows that for very close trajectories, arguments provided for $\epsilon=0$ do not apply. We will illustrate how this phenomenon may lead to the emergence of MMOs.  
\end{rem}

The existence of a travelling wave solution is stated now.
\begin{cor}\label{cor-simple_periodic_TW}
Assume the hypothesis are as in \cref{thm-simple_periodic-eps>0}. Then $\{ S_j \}_{j \geq 1}$ has a $\alpha$-periodic traveling wave solution.
\end{cor}
\begin{proof}
Since the chain is homogeneous in $i$, if every node starts in an initial condition corresponding to the periodic solution for $S_1$, the trajectories will be the same for each $i$, with a delay between the node $i$ and the node $i+1$ corresponding to the time needed for the node $i$ to reach the semi-axis $u=0,v<0$. 
\end{proof}
The next proposition deals with the stability of the traveling wave solution.  Let $(\varphi_j(t)), 1\leq j \leq N$ the traveling wave solution exhibited in the corollary \ref{cor-simple_periodic_TW}. For any initial condition of the system \cref{E:main_model}, for each $j\in \{1,..,N\}$, we denote by $(t^n_j)_{n\in \N}$ the sequence of times at which the node $S_j$ receives a kick. The following proposition holds

\begin{prop}\label{qual-simple_periodic_TW}
Assume  the hypothesis are as in \cref{thm-simple_periodic-eps>0}. Then the TW is locally stable in the following sense, $\forall j \in \{1,...,N\}$, 
\[ \lim_{n\rightarrow +\infty} t^n_{j+1}-t^n_j=\alpha\]
and
\[  \lim_{n\rightarrow +\infty} S_j(t^n_j)=(u^*,v^*).\]
\end{prop}

\begin{figure}
\begin{center}
\includegraphics[width=7cm, height=5cm]{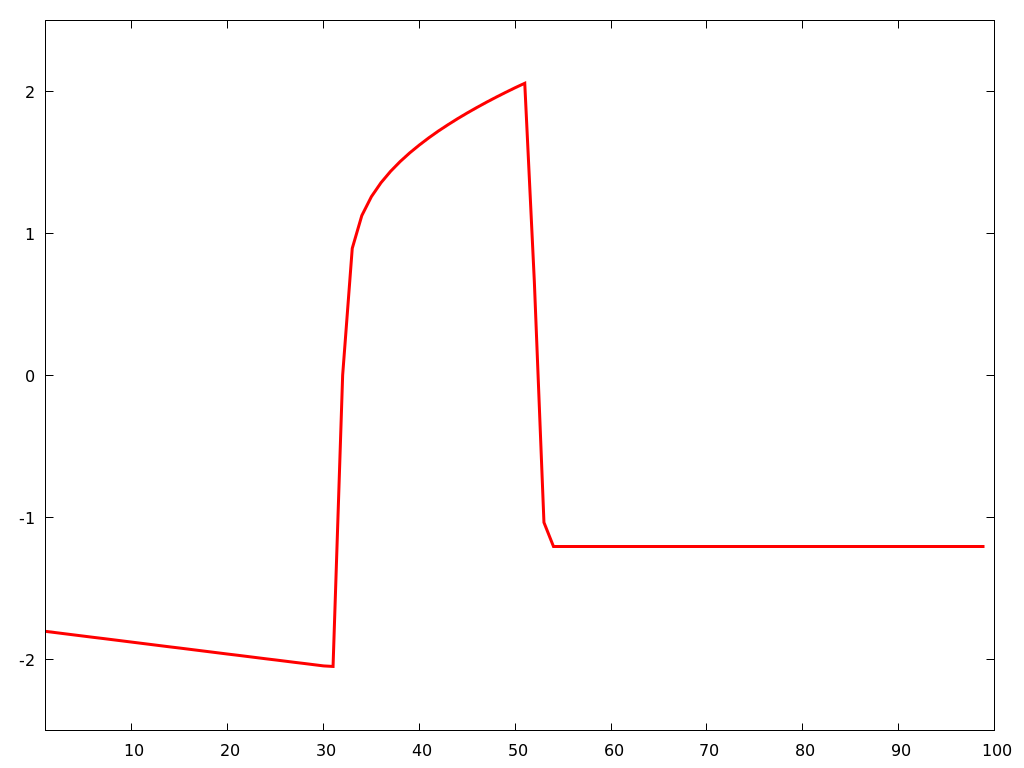}
\caption{Travelling wave profile. Represented  $(u_i)_{1\leq i\leq 100}$ at time $t=110$. $\alpha=50$.}
\label{fig:u-x-t=155-al=50}
\end{center}
\end{figure}

%

We conclude this discussion with a heuristic calculation of the velocity (and related quantities) for the traveling wave. Notice that in this simplest case, one can deduce an explicit estimate on the speed of the resultant $\alpha$-periodic traveling wave, or equivalently, the time of propagation for pulses presented to the front node. For any fixed cell, the time lag between stimulus arrival and consequent excitation of its right neighbor is approximately equal to the amount of time that it takes for a trajectory to travel from $u=c$ to $u=K$. Accordingly, we obtain
\begin{equation*}
T_e=\int_{t_1}^{t_2} \, dt = \int_{c}^{K}\frac{du}{f(u)-v} .
\end{equation*}
Taking the additional simplifying assumption of starting from rest allows one to replace $v$ with $f(c)-A$, and this yields the formula
\begin{equation*}
T_e\simeq \epsilon\int_{c}^{K}\frac{du}{f(u)-f(c)+A} .
\end{equation*}
In addition, notice that the time needed to excite the whole chain is $N\times T_e$. We emphasize that this is the amount of time taken by a single pulse to travel through chain front to back. Given the integral units of the spacial coordinate, all of this allows us to deduce that the speed of the traveling wave is $1/T_e$. And lastly, in accordance with the definition of TW detailed in \cref{subsec:basic_props_FHN}, one can take the parameter $\beta$ to be any element of $T_e + \alpha \mathbb{Z}$; without loss of generality, we think of $\beta$ as the smallest positive element of this set -- namely, $T_e$.

\section{Mixed-mode oscillations at $S_1$}
\label{sec:mmo_at_S_1}
It should be noted that for $\alpha_1<\alpha<\alpha_0$, prior to the simple MMO regime described in \cref{tab:parameter_values-numerics} and established herein, continuous dependence on parameters implies a highly-complex transition from the simple $\alpha$-periodic trajectory to complicated oscillations; recall \cref{fig:result_diagram}. For instance, this can be realized via canard trajectories that either follow the unstable branch of the critical manifold for some time prior to making a fast jump (a path that possibly includes threshold crossing) or otherwise canard trajectories that loop directly back toward the equilibrium without crossing threshold. It is equally worth noting that a possibly very complicated combination of small and large oscillations can ensue; see \cref{fig:u-v-al=8p4}. A detailed analysis of the bifurcation structure in this system will be the subject of subsequent investigations. For the present discussion, we establish the existence of simple MMO rigorously and then conclude this section with a few numerical investigations of the bifurcation cascade.

\begin{prop}\label{thm-two_loops-one_big_one_small}
Assume the condition on $A$ as in the hypothesis of \cref{thm-simple_periodic-eps=0}. There is an interval of periods $(\alpha_2,\alpha_1)$ such that for each $\alpha$ in this interval, the dynamics of $S_1$ feature a $2\alpha$-periodic trajectory with one large loop that crosses threshold at $u=0$ and one small simple loop that does not. (See \cref{fig:u-v-al=8}; this is the simplest mixed-mode oscillation scenario.) 
\end{prop}

\begin{figure}
\begin{center}
\includegraphics[width=7cm, height=5cm]{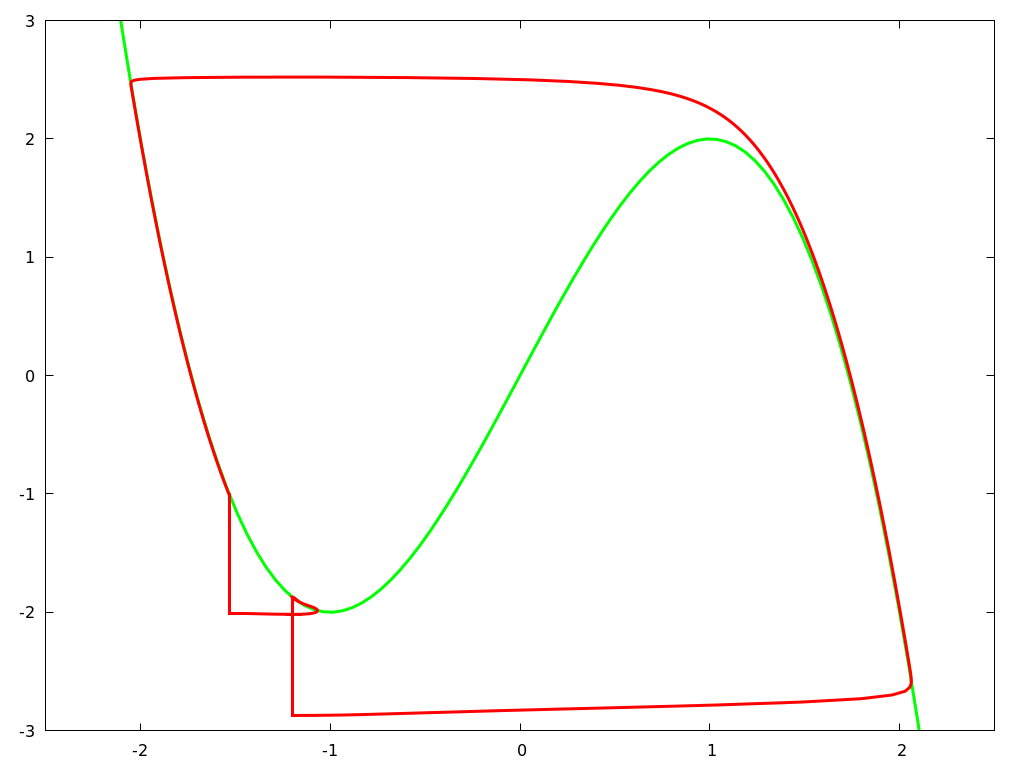}
\caption{Evolution of  $\varphi_1=(u_1,v_1)$ as a function of time. $\alpha=8$. The small oscillation corresponds to the attraction to the stable point: the kick arrives a little bit too soon. Then, we stay there until the next kick. This is a numerical realization of the result claimed in \cref{thm-two_loops-one_big_one_small}.}
\label{fig:u-v-al=8}
\end{center}
\end{figure}

\begin{figure}
\begin{center}
\includegraphics[width=7cm, height=5cm]{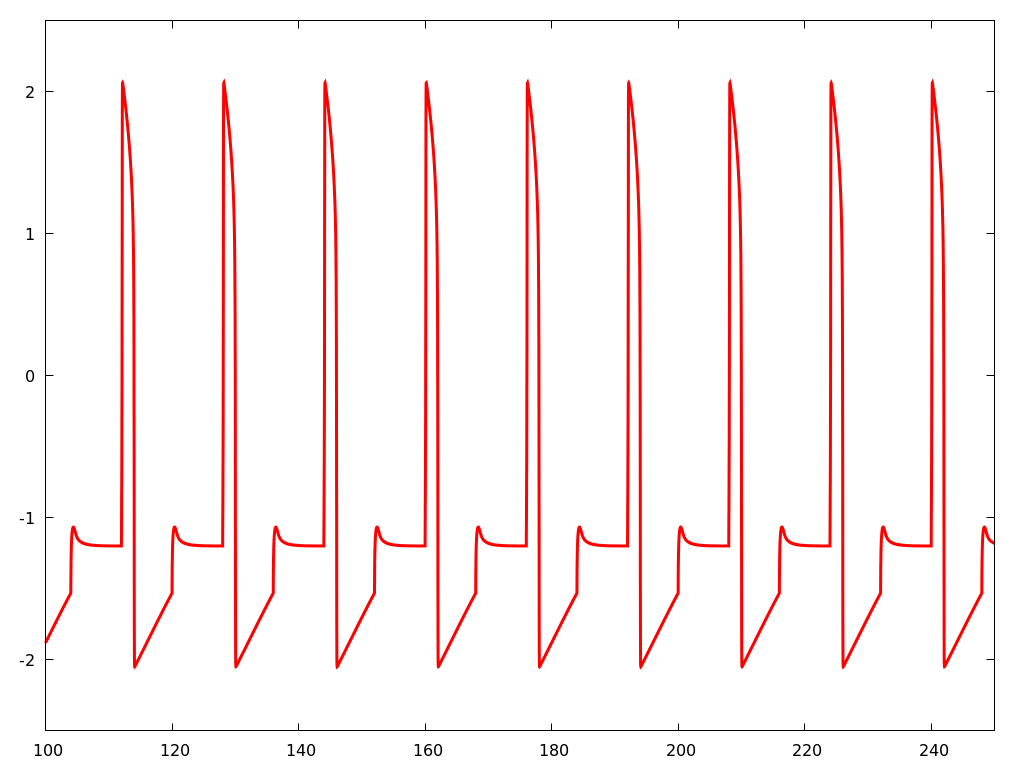}
\caption{Evolution of  $u_1$ as a function of time. $\alpha=8$.}
\label{fig:u-t-al=8}
\end{center}
\end{figure}

\begin{figure}
\begin{center}
\includegraphics[width=7cm, height=5cm]{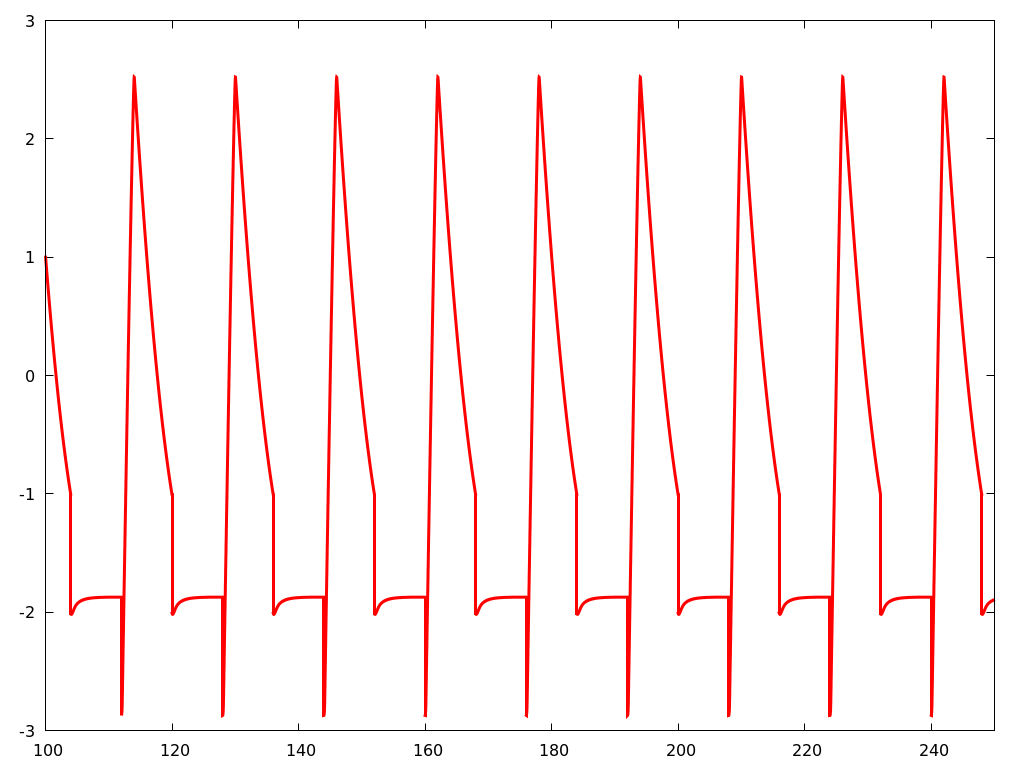}
\caption{Evolution of  $v_1$ as a function of time. $\alpha=8$.}
\label{fig:v-t-al=8}
\end{center}
\end{figure}

\begin{rem}
Altogether, this scenario results in an overall effect of doubling the depolarization period for the first cell. In this case, if $2\alpha > \alpha_0$, it follows that the $\{ S_j \}_{j \geq 2}$ behaves (asymptotically) according to the simple low-frequency paradigm of periodic traveling wave. See \cref{subsec:propagation_simple_mixed} and \cref{fig:u-x-t=110-al=8}.
\end{rem}

\begin{proof}[Proof of \cref{thm-two_loops-one_big_one_small}]
First note that there exists a trajectory $\psi(t)$ that originates on the left branch of the critical manifold and, right after a vertical kick (by $-A$), features a fast horizontal trajectory that crosses the line $u=c$ above the line $v=-2$. This trajectory then turns to the left and is simply attracted to the equilibrium (as $t \to \infty)$. In fact, there is a portion $\mathcal{C}_0$ of the critical manifold consisting of points with $v$-coordinates slightly higher than the $v$-coordinate of the origination point of this special trajectory, all of which feature precisely this simple dynamics.

Suppose $B_{\delta}$ is a ball about $(c,f(c))$. It follows not only that all trajectories starting on $\mathcal{C}_0$ will end up in $B_{\delta}$, but also that the time it takes before they enter (and never leave thereafter) is less than the time it takes for $\psi(t)$ to enter it. This is because $\psi(0)$ is the lowest ($v$-coordinate) point on $\mathcal{C}_0$.

We now deduce some necessary explicit time estimates for trajectories of the system. The time that it takes for a typical trajectory to change its $v$-coordinate from $v_1$ to $v_2$ can be computed by using the dynamics on the critical manifold and appealing to the formula
\[
t_2-t_1=\int_{v_1}^{v_2}\frac{dv}{f^{-1}(v)-c}+O(\epsilon)=\int_{f(u_1)}^{f(u_2)} \left( \frac{f'(u)}{u-c} \right) \, du+O(\epsilon), 
\]
in tandem with the change of variables $v=f(u)$. Explicit calculations are straightforward for the example $f(u)=3u-u^3$ and yield
\[
t_2-t_1 = 3(1-c^2) \left[ \ln(u_2-c)-\ln(u_1-c) \right] - 6c(u_2-u_1)-\frac{3}{2} \left( (u_2-c)^2-(u_1-c)^2 \right)+O(\epsilon).
\]
This formula is used in the calculations captured within \cref{fig:TimeEst-3}.

We argue for $\eps=0$. Assume $\alpha$ is large enough but not too large, so that a trajectory that ensues from $(c,f(c))$ ends up on $\mathcal{C}_0$ at time $t=\alpha$. It follows that there is a range of $\alpha$ such that for all $\delta$ sufficiently small, the ball $B_{\delta}$ about $(c,f(c))$ is mapped inside $\mathcal{C}_0$ after one kick and time $\alpha$ flow of the ODE. We close by noting that there exists a range of $\delta$ for which the time needed for $\psi(t)$ to enter $B_{\delta}$ after one kick plus ODE flow is much smaller than smallest of the $\alpha$ in the aforementioned interval. (See \cref{fig:TimeEst-3}.)

Altogether, this means that there is a closed ball $B_{\delta}$ about $(c,f(c))$ that is mapped within itself by the sequence of two applications of kick plus time-$\alpha$ flow. By application of the Brouwer fixed point theorem, we are guaranteed the existence of a fixed point to the discrete time map consisting of the conjunction of kicks and flows just described. This property then assures the existence of the aforementioned $2\alpha$-periodic solution to the dynamics in $S_1$ under the $\alpha$-periodic forcing.
\end{proof}

\begin{rem}\label{rem-alpha_less_than_alpha2}
When $\alpha$ is only slightly smaller than $\alpha_2$, the observed phenomenology at $S_1$ is straightforward despite not being of MMO type. The arrival of the first kick is indeed too soon on the left branch of the critical manifold -- in fact, so soon that the trajectory is not kicked below the equilibrium point $(c,f(c))$ on the graph of $f$; accordingly, the solution undergoes a simple fast trajectory to the right and then continues {\it down} the left branch of the critical manifold until the second kick; the latter then causes a vertical drop sufficient for the trajectory to clear the critical manifold and cross threshold. Just like in the case of simple MMO, this period-$2\alpha$ solution has immediate consequence for $S_2$, which evidently sees a lower-frequency input. In fact, the chain $\{ S_j \}$ for $j \geq 2$ will entrain to the $2\alpha$-periodic TW that is guaranteed for that forcing period by the analysis in~\cref{sec:simple_TW}, because the numerical values in~\cref{tab:observed-critical_period_values} indicate $2\alpha > \alpha_0$ when $\alpha \approx \alpha_2$. Of course, for other choices of the system parameters (i.e., different from those declared in~\cref{tab:parameter_values-numerics}), more complicated propagation scenarios might occur.
\end{rem}

\begin{figure}[!h]  
\begin{center}
\includegraphics[width=7cm, height=5cm]{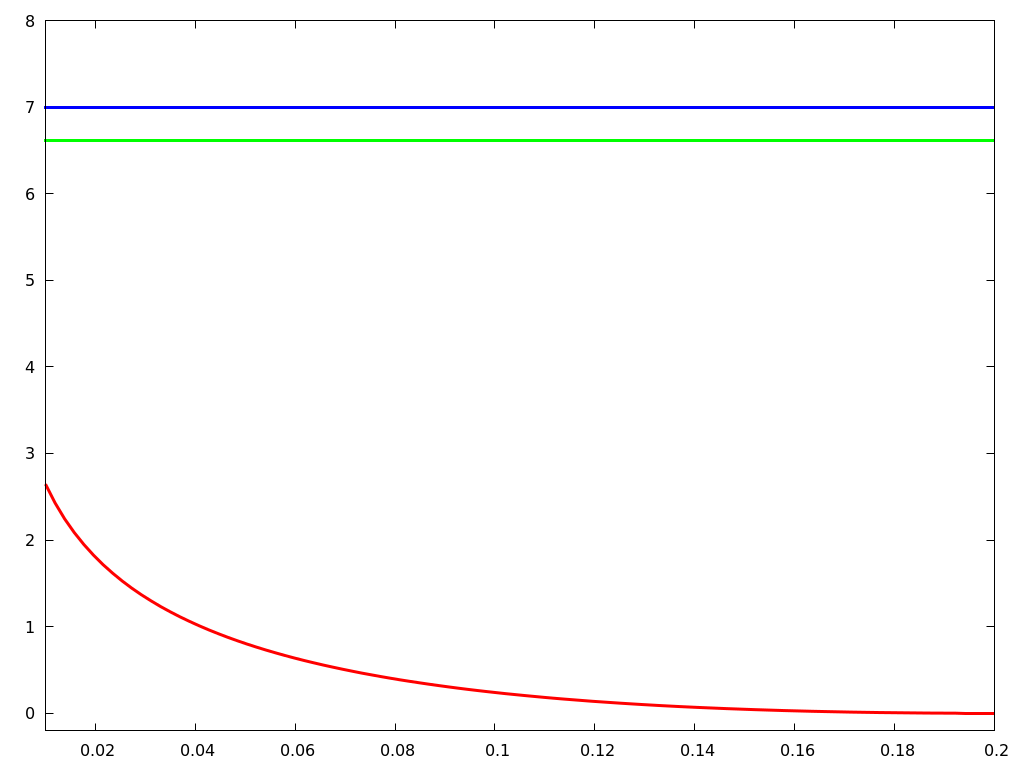}
\caption{In red, explicit time estimate in case of $\eps=0$, for $f(u)=3u-u^3$, $c=-1.2$ and $A=1$, of the time necessary for the trajectory issued from $(-1,-2)$ (the local minimal point of the graph of $f$) to reach the point $(c+\delta,f(c+\delta))$, as a function of $\delta$; in green is the amount of time necessary for the trajectory issued from $(c,f(c))$ to reach the point with $v$-coordinate $f(c)+A$ on the left branch of the critical manifold (the latter point may be taken as the definition of the other boundary point of $\mathcal{C}_0$; and blue, the time necessary for $(c,f(c))$ to reach the point on the left branch of the critical manifold possessing $v$-coordinate $-2+A$. The separation between between the green and purple establishes the existence of a range of $\alpha$ for which the estimates claimed in the proof of~\cref{thm-two_loops-one_big_one_small} are feasible.}
\label{fig:TimeEst-3}
\end{center}
\end{figure}

\begin{figure}
\begin{center}
\includegraphics[width=7cm, height=5cm]{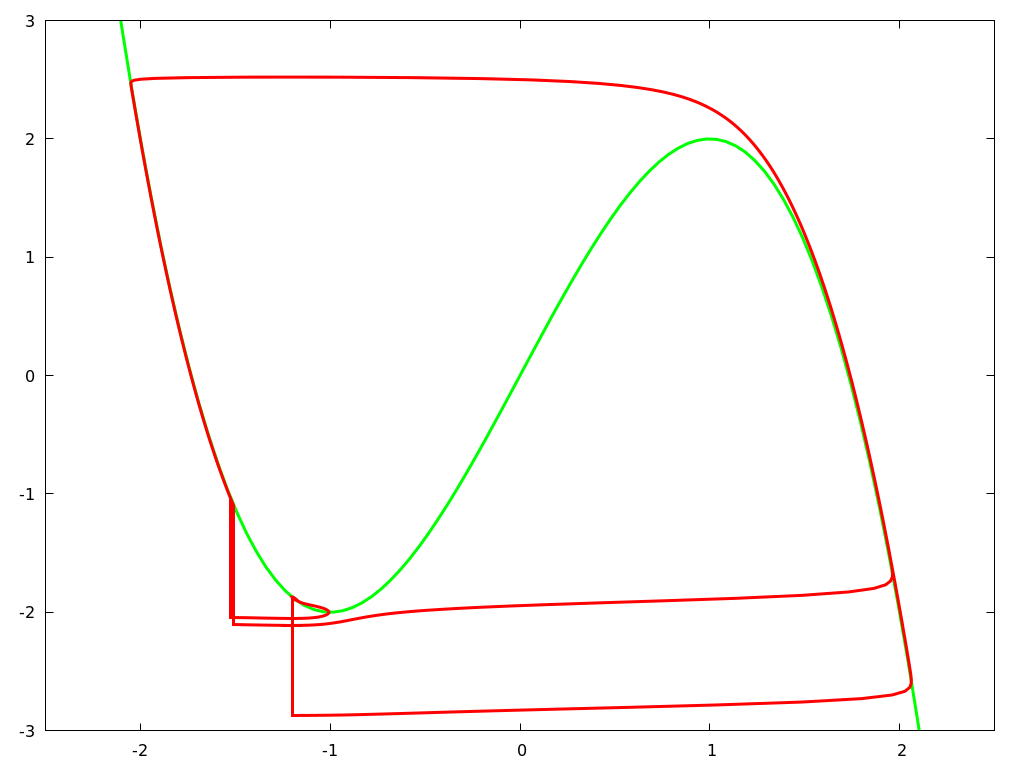}
\caption{Steady-state evolution of $\varphi_1=(u_1,v_1)$ as a function of time. $\alpha=8.3$. Here $\alpha$ is in the interval $(\alpha_1,\alpha_0)$. After a kick and a first big loop the solution is not attracted to the stable point; instead, it goes for another large oscillation. However, the time needed to reach the right portion of the slow manifold is longer on this second loop because of the velocity dissipation when the trajectory passes close to the local minimal point (-1,-2) of the graph of $f$. When the next kick arrives, the solution is further up along the left branch of the slow manifold and consequently enters a neighborhood that is directly attracted to the stable point; in turn, this gives rise to an oscillation with a much smaller amplitude. As a result, the steady-state here appears to be a periodic (or nearly so) trajectory with two large and one small oscillations.}
\label{fig:u-v-al=8p3}
\end{center}
\end{figure}

\begin{figure}
\begin{center}
\includegraphics[width=7cm, height=5cm]{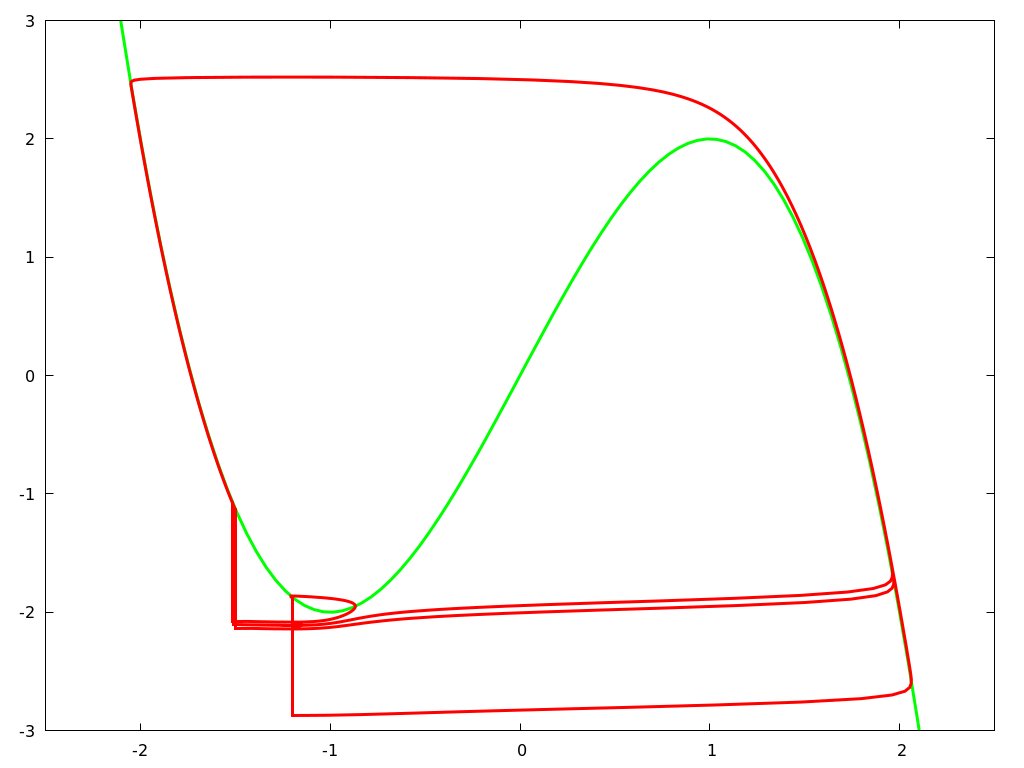}
\caption{Evolution of  $\varphi_1=(u_1,v_1)$ as a function of time. $\alpha=8.4$. The complexity of MMOs continues to grow as $\alpha$ is increased to $\alpha_0$: here, the steady state features three large oscillations and one small one. Some trajectories exhibit canard behavior.}
\label{fig:u-v-al=8p4}
\end{center}
\end{figure}

\begin{figure}
\begin{center}
\includegraphics[width=7cm, height=5cm]{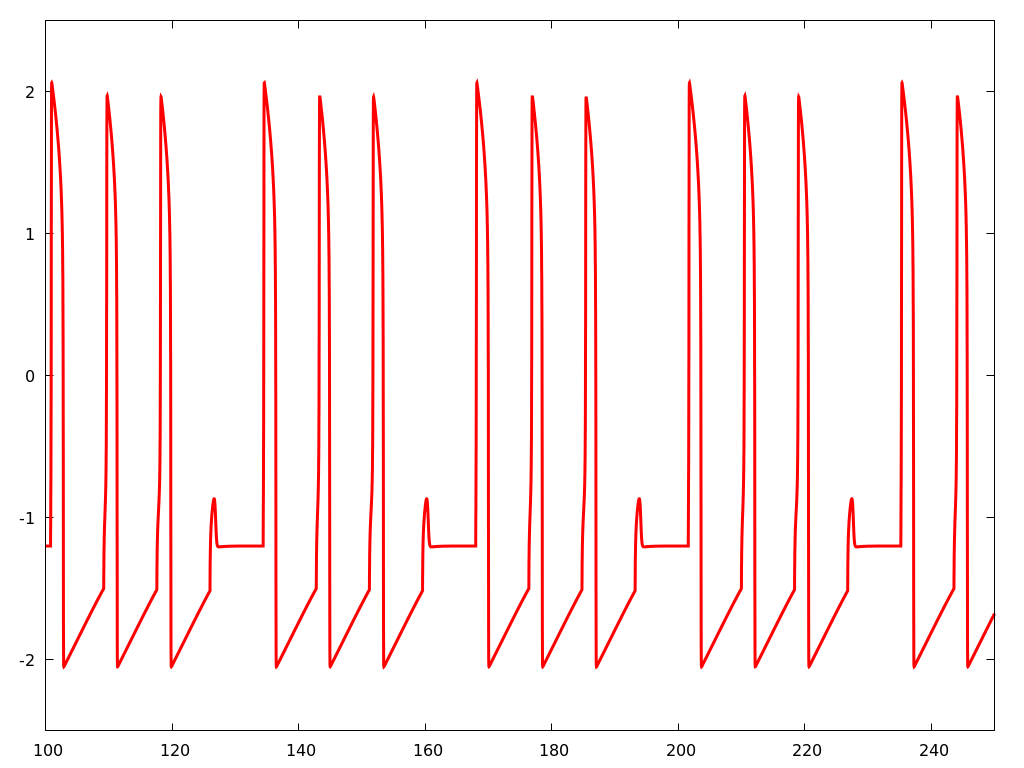}
\caption{Evolution of  the voltage variable $u_1$ at steady-state as a function of time. $\alpha=8.4$.}
\label{fig:u-t-al=8p4}
\end{center}
\end{figure}

\begin{figure}
\begin{center}
\includegraphics[width=7cm, height=5cm]{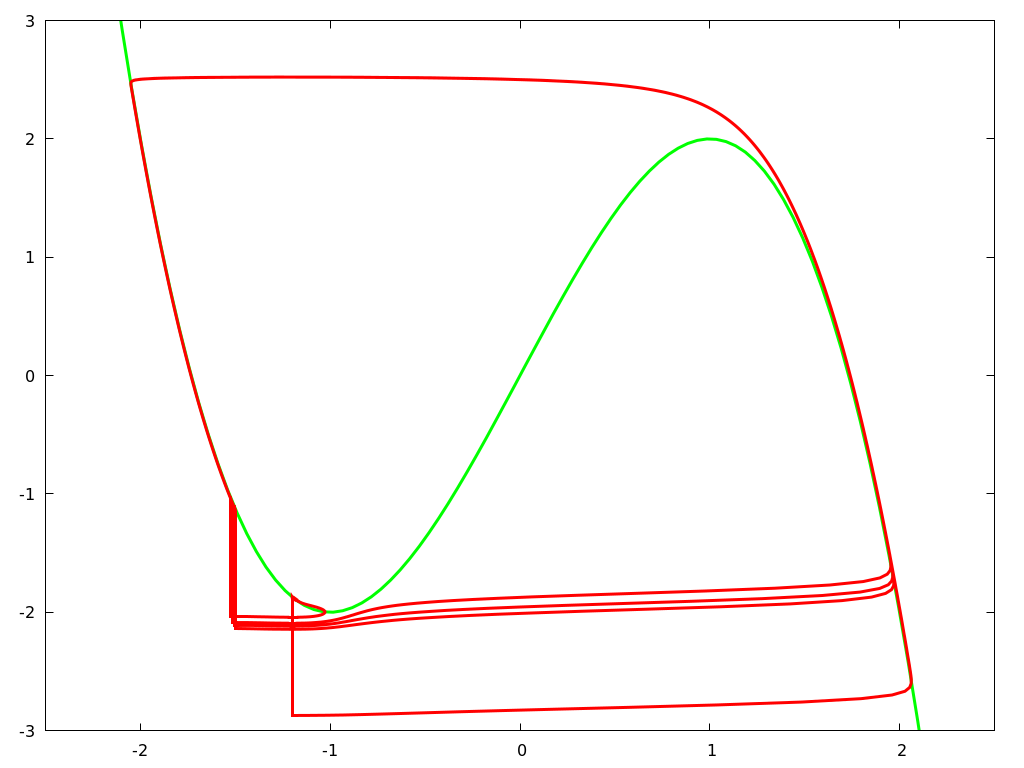}
\caption{Evolution of  $\varphi_1=(u_1,v_1)$ as a function of time. $\alpha=8.41$. Four large and one small oscillations at steady state. 
}
\label{fig:u-v-t-al=8p41}
\end{center}
\end{figure}

\begin{figure}
\begin{center}
\includegraphics[width=7cm, height=5cm]{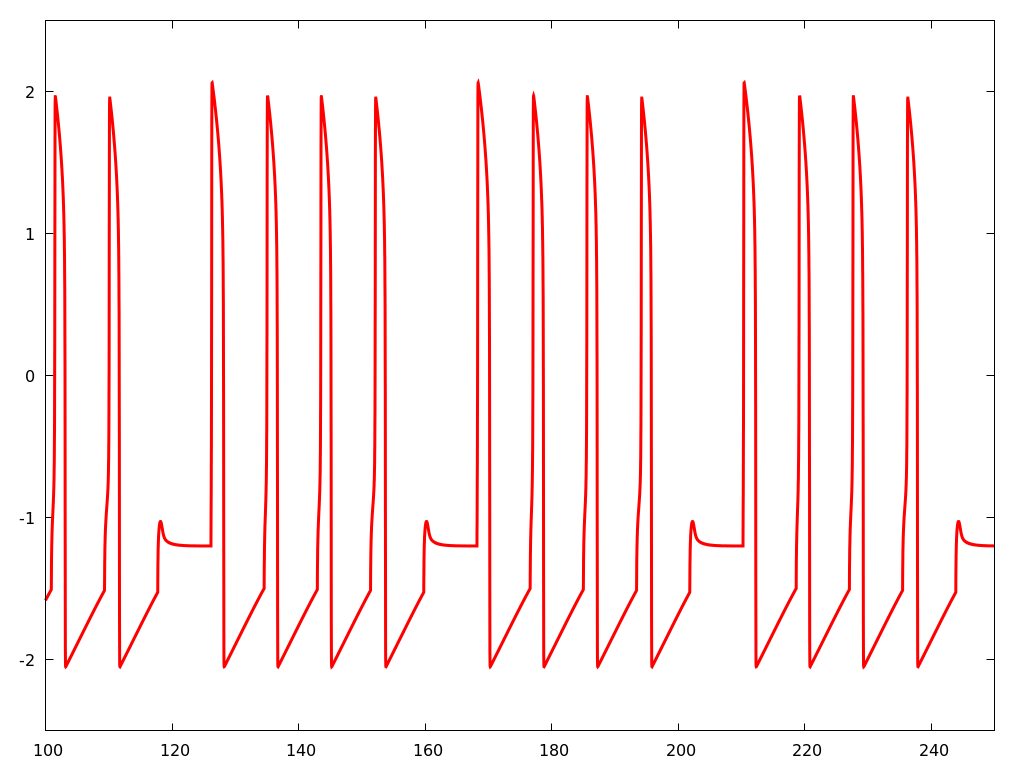}
\caption{Evolution of  $u_1$ as a function of time. $\alpha=8.41$.}
\label{fig:u-t-al=8p41}
\end{center}
\end{figure}

\begin{figure}
\begin{center}
\includegraphics[width=7cm, height=5cm]{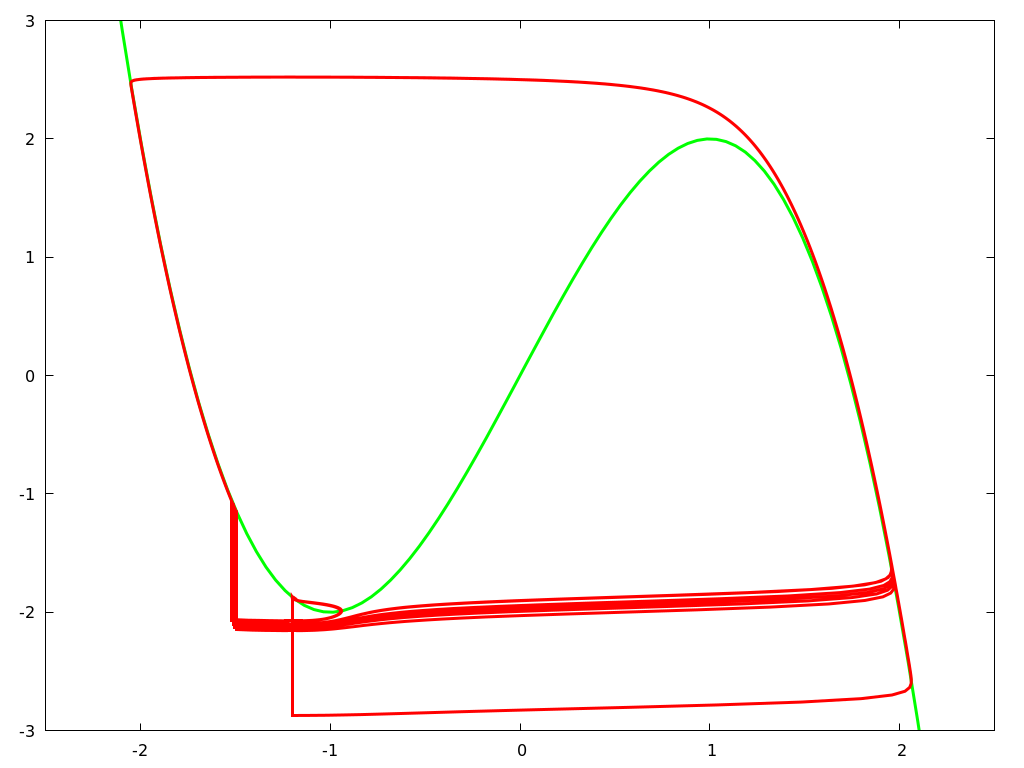}
\caption{Evolution of  $\varphi_1=(u_1,v_1)$ as a function of time. $\alpha=8.45$.  Five large and one small.}
\label{fig:u-v-t-al=8p45}
\end{center}
\end{figure}

\begin{figure}
\begin{center}
\includegraphics[width=7cm, height=5cm]{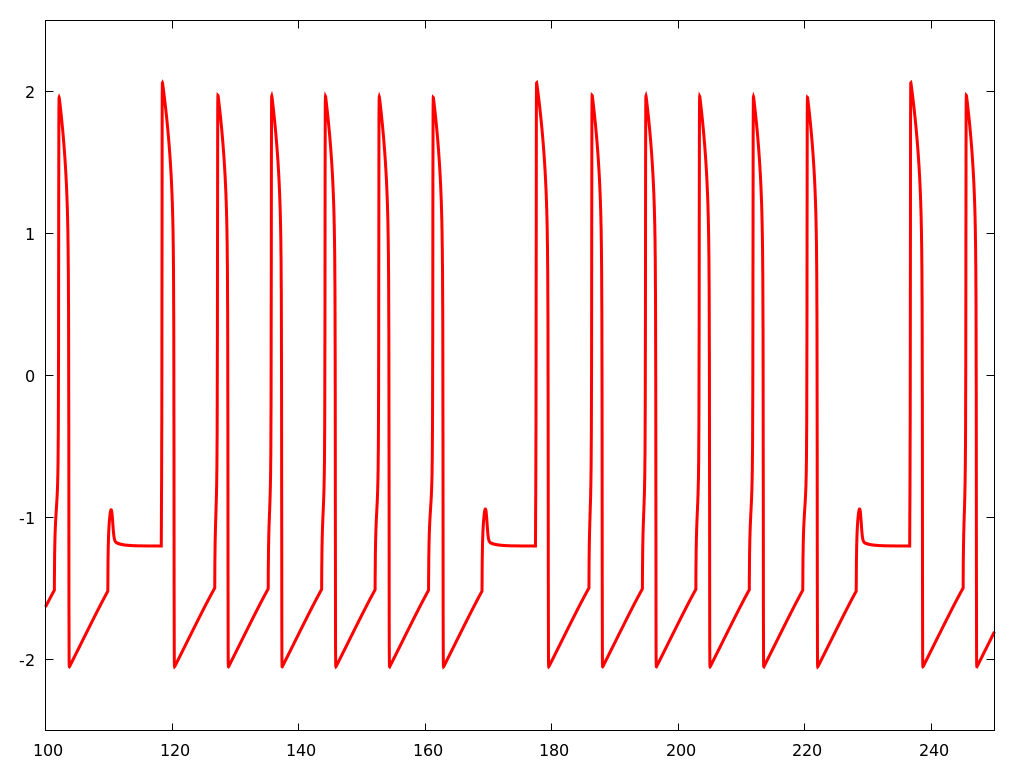}
\caption{Evolution of  $u_1$ as a function of time. $\alpha=8.45$.}
\label{fig:u-t-al=8p45}
\end{center}
\end{figure}

In the interval between $\alpha_1$ and $\alpha_0$, the system can feature longer sequences of regular depolarizations interspersed by occasional quiescent loops. This is an interesting and mathematically complicated parameter regime. We share here some numerical observations together with our proposed geometric mechanism for this behavior.
\begin{rem}
One of the remarkable phenomena occurring in the regime $\alpha_1 \leq \alpha \leq \alpha_0$ is that, after each large-amplitude oscillation, the values of $v$ at which the trajectory for $S_1$ receives a kick form an increasing sequence until attraction to the stable equilibrium causes a small oscillation. This gives rise to an effective periodicity with MMOs, that is, a rhythmic train consisting of sequences of consecutive large-amplitude oscillations interrupted by one small one.
\end{rem}

\noindent {\it Proposed mechanism.}
The reader should refer to \cref{fig:u-v-t-al=8p41,fig:u-v-t-al=8p45,fig:u-v-differentcolors}. With reference especially to the third figure here, there is a region below the critical manifold $v=f(u)$ in which trajectories (these go mostly left-to-right) take progressively longer to go across from the left part of phase space over to the right branch of  the manifold; this tendency agrees with the vertical ordering of these trajectories, with the higher ones taking a longer time. This phenomenon is most drastically observed when trajectories of the system track the critical manifold beyond the critical point at $u=-1$. In particular, given two such trajectories -- one below the other -- the left-to-right crossing time for the higher one is larger than the combination of crossing and ascent times for the lower one, which is to say that if the system were to travel (say, from initial condition $u(0)=-1$) along these trajectories, it would take longer for it to reach the right jump point on the higher trajectory than on the lower one.

\begin{figure}
 \begin{center}
\includegraphics[width=10cm, height=8cm]{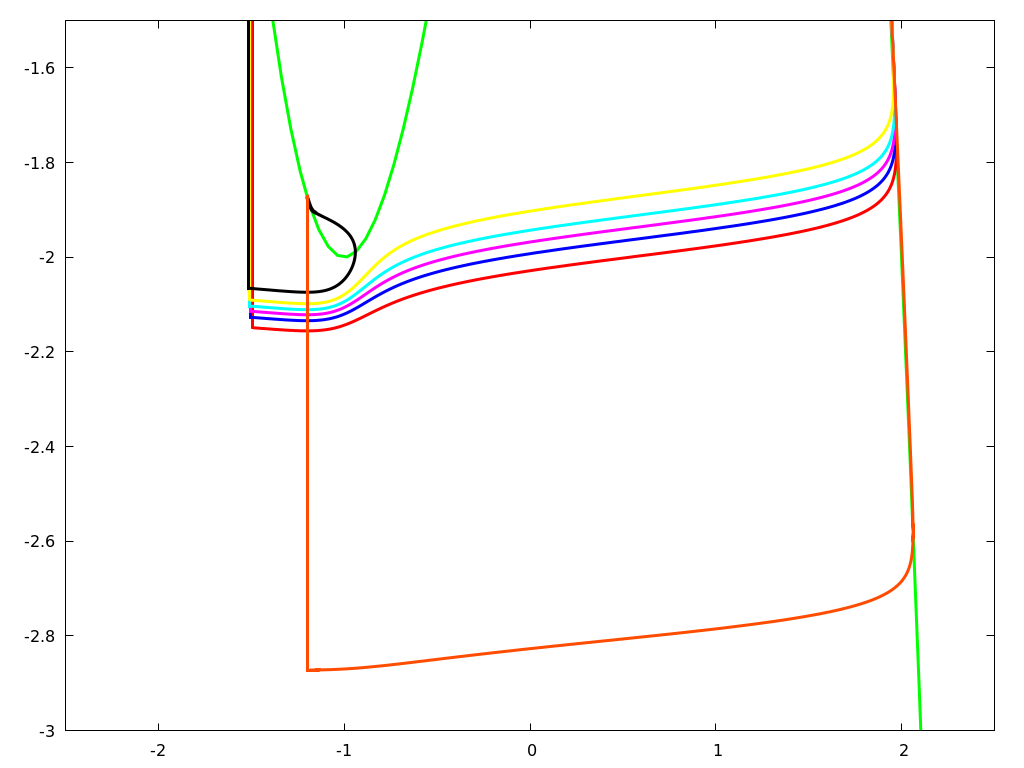}
\caption{A phase plot of a single trajectory of $(u_1,v_1)$ for $\alpha=8.45$, specifically, its successive passes after successive kicks (shown in various colors, in order) progressively closer to the left extremal point of the critical manifold. This is due to a speed dissipation effect that grows with proximity to the minimal point of the graph of~$f$. This figure suggests the following mechanism: after six large oscillations (with depolarization) corresponding to six kicks, at the seventh kick, the trajectory is trapped toward the attractive equilibrium with a small oscillation. The time afterwards needed to reach a small neighborhood of the steady state is much shorter than $\alpha$; consequently, at the next kick the trajectory is not distinguishable from the one following the first kick. Hence, this mechanism gives rise to a 6-large-1-small MMO.}
\label{fig:u-v-differentcolors}
\end{center}
\end{figure}

Now, let us assume that the system is kicked from rest at $(c,f(c))$. After a large oscillation, provided that the next kick arrives when the trajectory is in the downward (slow) swing, the new (post-kick) $v$ coordinate of the trajectory is bounded below (which is to say, above) $v=f(c)-A$. At this point, either the trajectory features a small oscillation with attraction to $(c,f(c))$, or it features another large oscillation (but one whose lower portion is at a level above the first one). It seems then that the MMO phenomenology observed in, e.g., \cref{fig:u-v-t-al=8p45}, is owed to the fact that when not attracted directly to $(c,f(c))$, the trajectory finds itself precisely in the region of phase space previously described. A sequence of large oscillations ensues, each one featuring a left-to-right crossing above the prior, until the the kick arrives sufficiently early to induce a small loop with attraction to $(c,f(c))$. In effect, this last step renews the system and the process starts over.

\bigskip

We couple this observation with a description of the phenomenology at $S_1$ as the forcing period is decreased from $\alpha_0$ to $\alpha_1$: In principle, there can be a very high number of large loops prior to a small resetting oscillation. This number is larger for $\alpha \approx \alpha_0$ and steadily decreases, as $\alpha$ decreases to $\alpha_1$ (see \cref{fig:u-v-al=8p3,fig:u-v-al=8p4,fig:u-v-t-al=8p41,fig:u-v-t-al=8p45}). Notice that, as noted in \cref{thm-two_loops-one_big_one_small}, for $\alpha$ in the interval $(\alpha_2,\alpha_1)$, there is just one large loop paired to a small one.

\begin{rem}
Understanding the extent of this complex phenomenology constitutes an involved mathematical problem related to various prior work and deserves a specially dedicated investigation beyond the scope of our descriptive exposition here.  A few examples of mathematical analysis related to the present problem are provided in the context of  three-time-scale equations, see \cite{Kru-2008,Kru-2008-2,Mae-2014} and also \cite{Amb-2021}. In there, as in here, solution trajectories constrained to follow a canard (and thereby result in relaxation oscillations) will instead diverge from this expected path, thereafter tracking in the  opposite direction. This behavior is captured nicely  by the computable solutions of a 2d-integrable system. 


\end{rem}

\section{Case studies of propagation and filtering of mixed-mode oscillations beyond $S_1$}
\label{sec:propagation_MMO}
\subsection{A simple mechanism for filtration}
\label{subsec:propagation_simple_mixed}
Assume first that $A$ is sufficiently large for \cref{thm-two_loops-one_big_one_small} to apply at $S_1$; furthermore, assume that $\alpha \in (\alpha_2 , \alpha_1)$. The dynamics of $S_1$ is then asymptotically attracted to a simple mixed-mode $2\alpha$-periodic oscillation that features depolarization only once in its period. It is immediate that $S_2$ receives forcing at half the frequency of the original input (the period is doubled to $2\alpha$). Thereafter, if $2\alpha$ happens to exceed $\alpha_0$, the tail of the chain will support the corresponding stable $2\alpha$-periodic TW as opposed to a TW of period $\alpha$. This scenario -- both observed in our numerics and established rigorously -- presents a simple mechanism for the failure of wave propagation, a property observed in a variety of models in neuroscience. We refer to this phenomenology as {\it filtering}, since it presents an effective downshift in spiking frequency from the perspective of signal processing.

It is also worth noting that the presence of higher-complexity MMO for certain forcing periods (e.g., those between $\alpha_1$ and $\alpha_0$) provides a setting where this sort of failure of wave propagation (equivalently, wave mutation) can happen in a much more complicated manner. For instance, if $2\alpha$ is in the doubling interval $(\alpha_2,\alpha_1)$, the dynamics at $S_2$ will be asymptotic to a $4\alpha$-periodic trajectory; in turn, $S_3$ sees a sequence of Dirac kicks with period $4\alpha$. On the other hand, in principle one could have $2\alpha \in (\alpha_1,\alpha_0)$, whereby $S_2$ supports an MMO with some complexity, no longer featuring uniform spacing between consecutive kicks, but instead skipping kicks every once in a while. Of course, some of these possibilities may not be achievable in practice given the model parameters fixed in our numerical investigations.

\begin{figure}
\begin{center}

\includegraphics[width=7cm, height=5cm]{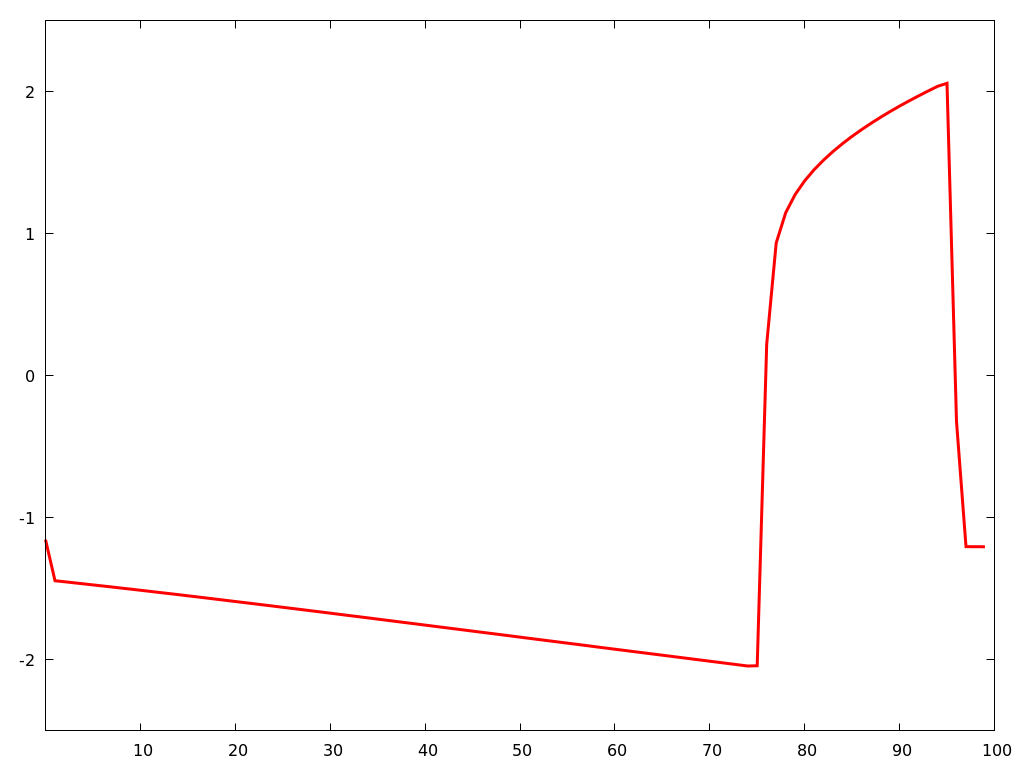}
\caption{Traveling wave of period $2\alpha$ for $\{ S_j \}_{j\geq2}$. Represented by $(u_j)_{1\leq j \leq 100}$ at time $t=100$. $\alpha=8$. Notice the kick between j=1 and 2 which will not generate a traveling wave but rather die at $j=2$ (compare with \cref{fig:u-x-t=155-al=50}); this is due to the fact that the wave profile starts at $j=2$.}
\label{fig:u-x-t=110-al=8}
\end{center}
\end{figure} 

%
%

\begin{figure}
\begin{center}
\includegraphics[width=7cm, height=5cm]{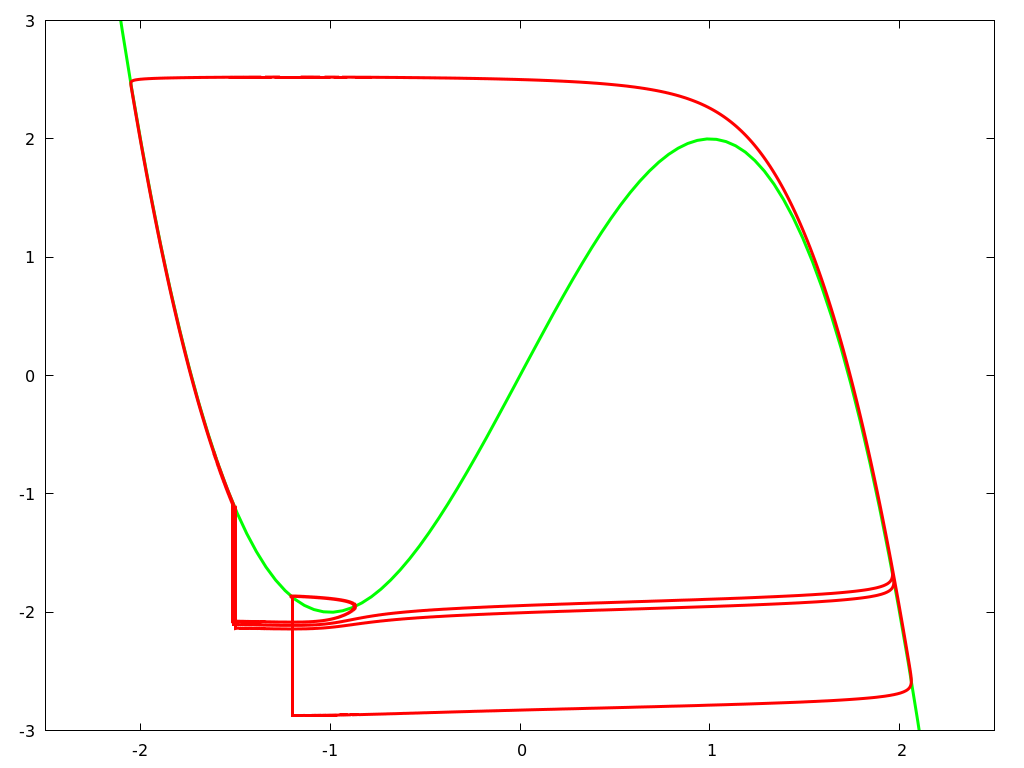}
\caption{Evolution of  $\varphi_2=(u_2,v_2)$ as a function of time. $\alpha=8.41$.}
\label{fig:u1-v1-t-al=8p41}
\end{center}
\end{figure}

\subsection{A few examples indicating even more-complex downstream filtering of mixed-mode oscillations}
\label{subsec:propagation_complex_mixed-numerics}
A cursory juxtaposition of \cref{fig:u-v-t-al=8p41,fig:u1-v1-t-al=8p41} reveals that $S_2$ presents a simpler MMO than $S_1$ at steady-state with $\alpha=8.41 \in (\alpha_1,\alpha_0)$. This is to say that the complexity of MMOs decreases down the chain. A similar observation follows from examination of \cref{fig:u-v-t-al=1p1,fig:u2-v2-t-al=1p1}, albeit in the regime $\alpha<\alpha_2$ which corresponds to the very high forcing frequency setting.

\begin{figure}
\begin{center}
\includegraphics[width=7cm, height=5cm]{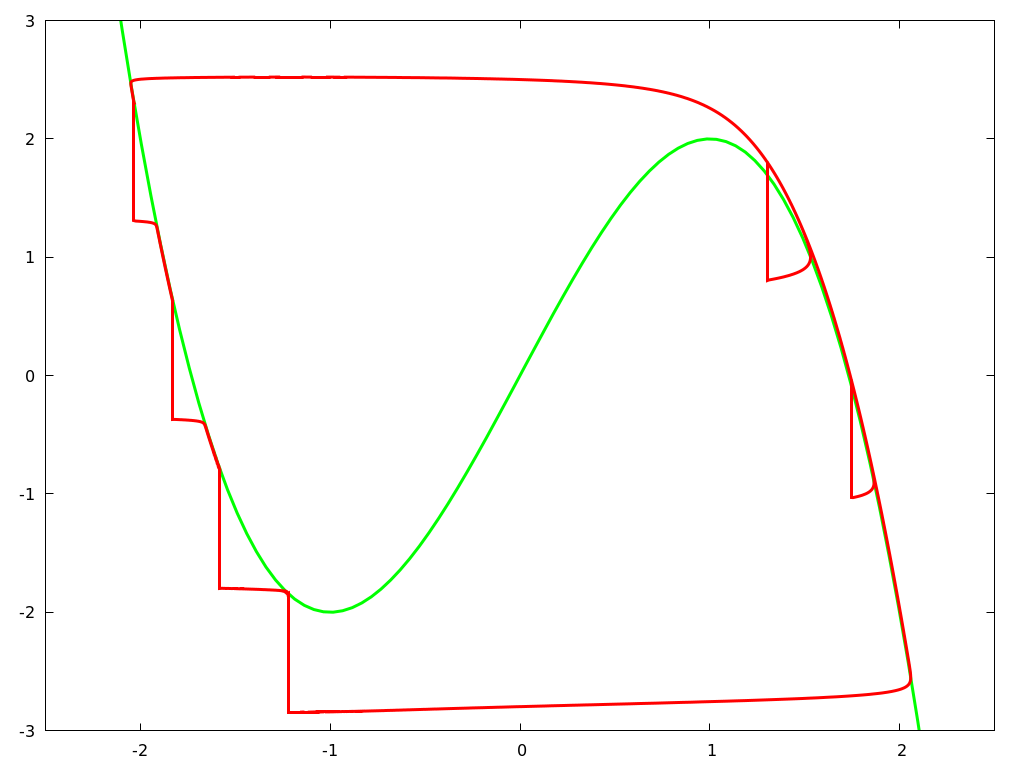}
\caption{Evolution of  $\varphi_1=(u_1,v_1)$ as a function of time. $\alpha=1$.}
\label{fig:u-v-t-al=1}
\end{center}
\end{figure}

\begin{figure}
\begin{center}
\includegraphics[width=7cm, height=5cm]{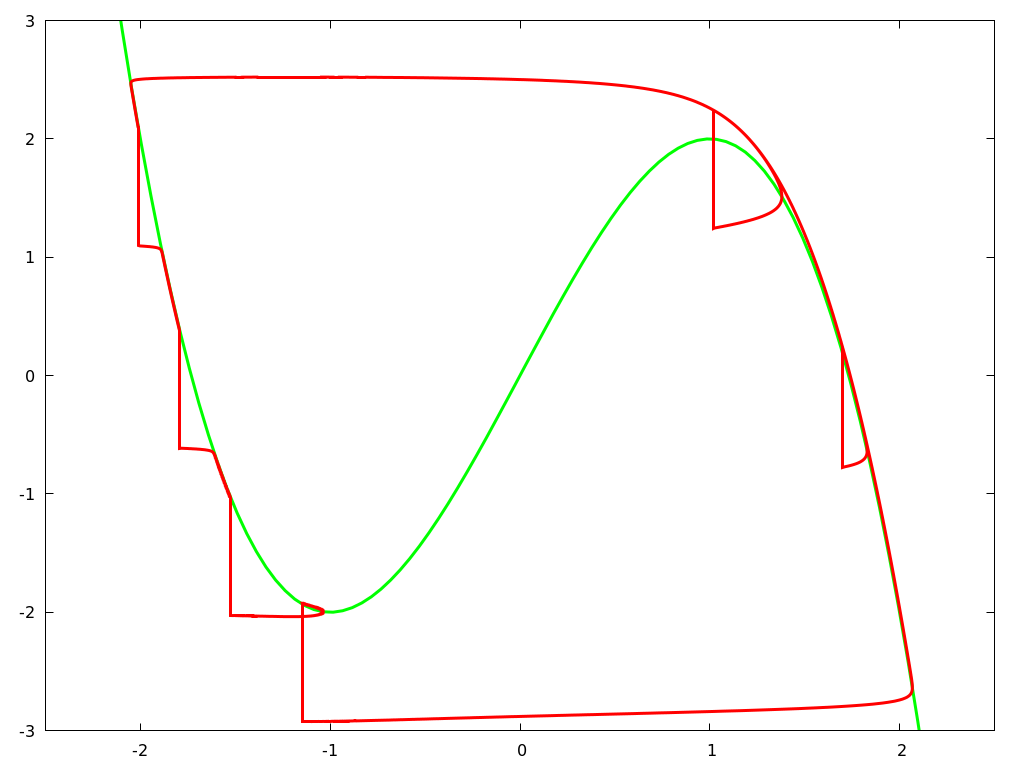}
\caption{Evolution of  $\varphi_1=(u_1,v_1)$ as a function of time. $\alpha=1.1$.}
\label{fig:u-v-t-al=1p1}
\end{center}
\end{figure}

\begin{figure}
\begin{center}
\includegraphics[width=7cm, height=5cm]{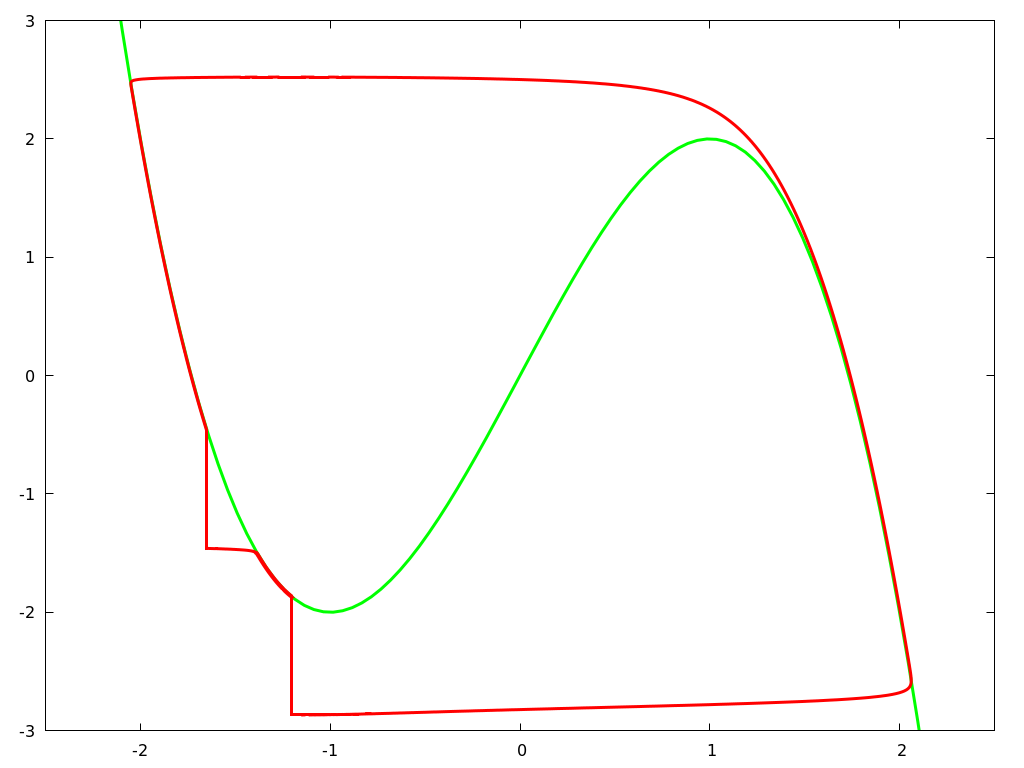}
\caption{Evolution of  $\varphi_2=(u_2,v_2)$ as a function of time. $\alpha=1.1$.}
\label{fig:u2-v2-t-al=1p1}
\end{center}
\end{figure}

The voltage dynamics in time allows another perspective on the filtration properties of the network. \Cref{fig:al=4cells1-2-3} illustrates an $\alpha$-to-$2\alpha$-to-$4\alpha$ doubling cascade through cells $S_1$, $S_2$, and $S_3$ for the instance in which $\alpha < \alpha_2$, followed by $2\alpha \in (\alpha_2,\alpha_1)$ (notice that \cref{thm-two_loops-one_big_one_small} applies here), whereafter $4\alpha > \alpha_0$. Starting with the third cell, the tail of the chain $\{ S_j \}_{j \geq 3}$ will feature a regular simple firing pattern with a $4\alpha$-periodic TW profile.
\begin{figure}
\begin{center}
\makebox[\textwidth]{\hfill
\includegraphics[width=3.5cm]{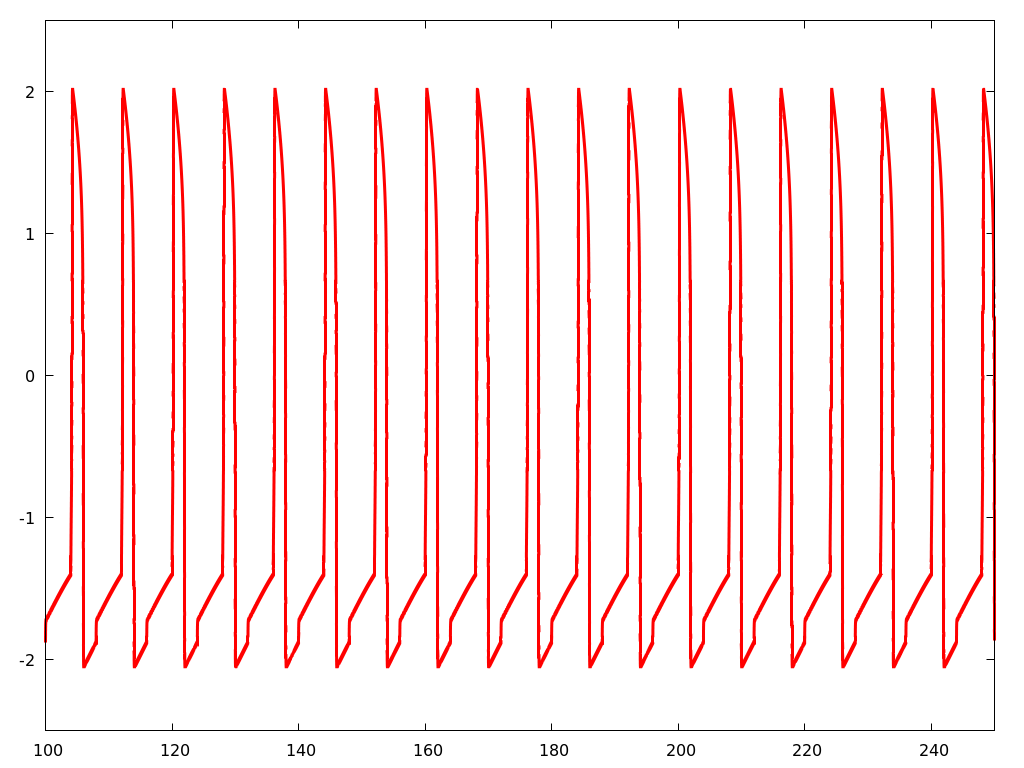}
\hfill
\includegraphics[width=3.5cm]{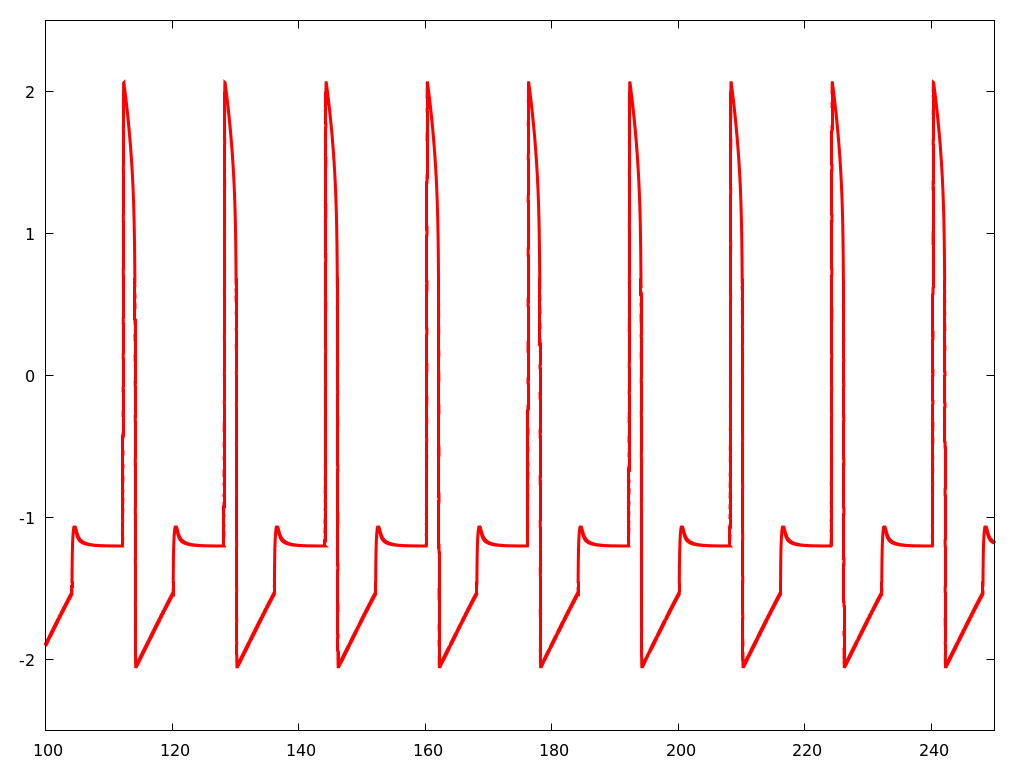}
\hfill
\includegraphics[width=3.5cm]{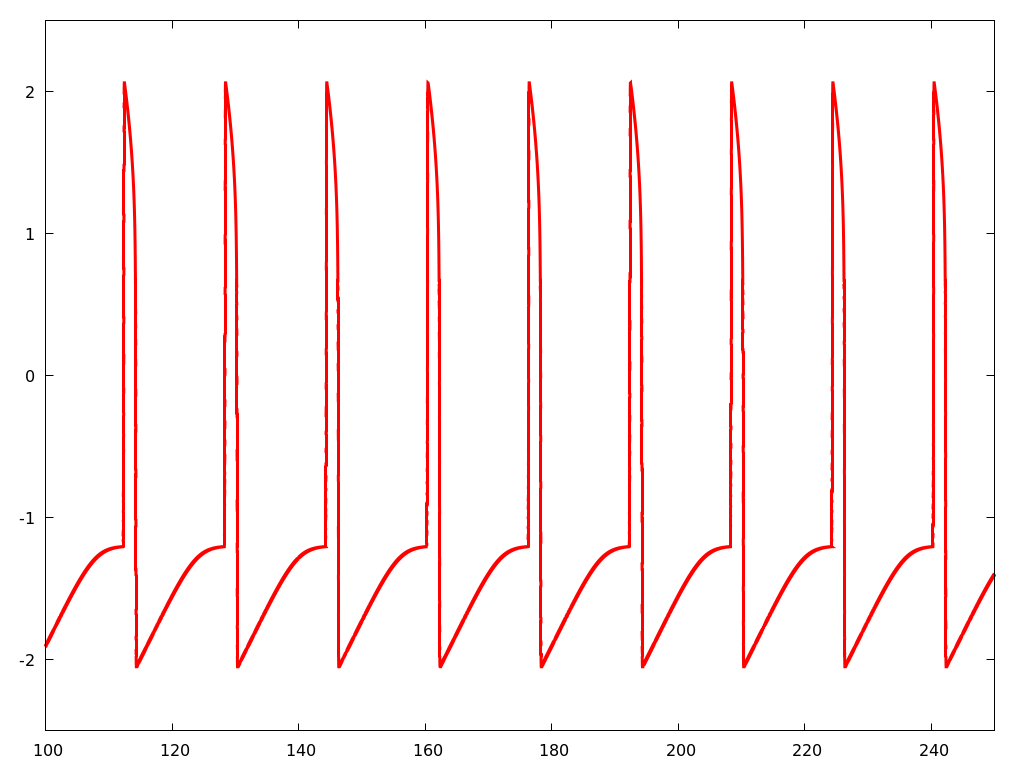}
\hfill
}
\caption{$\alpha=4$. 
\textbf{Left:} Voltage dynamics of $S_1$, simple periodic oscillation with period $2\alpha$ (twice the forcing due to the fact that every second stimulus arrives early, during the slow portion of the trajectory, so the dynamics continues toward the stable equilibrium after a brief rightward fast motion following that kick).
\textbf{Center:} Voltage dynamics of $S_2$, which receives periodic stimuli at period $2\alpha \in (\alpha_2,\alpha_1)$; this induces a simple MMO with one large-amplitude oscillation and one small-amplitude reset. The overall effective period is $4\alpha$ 
\textbf{Right:} Voltage dynamics of $S_3$ shows a simple $4\alpha$-periodic oscillation with one depolarization in each of its periods. \label{fig:al=4cells1-2-3}
}
\end{center}
\end{figure}

On the other hand, higher complexity of the MMOs at an upstream node (the situation where $2^j \alpha$ is finds itself in $(\alpha_1,\alpha_0)$ during the course of the induction) has further interesting consequences. This is illustrated in \cref{fig:al=4p2cells1-2-3}, where the third cell in the chain sees a burst-like behavior with four cycles per period, three of them `on' beats, and one of them a quiescent `off' cycle. Of note, the second cell behaves almost like this, but features a genuine MMO with a small-amplitude oscillation instead of the quiet phase. 
\begin{figure}
\begin{center}
\makebox[\textwidth]{\hfill
\includegraphics[width=3.5cm]{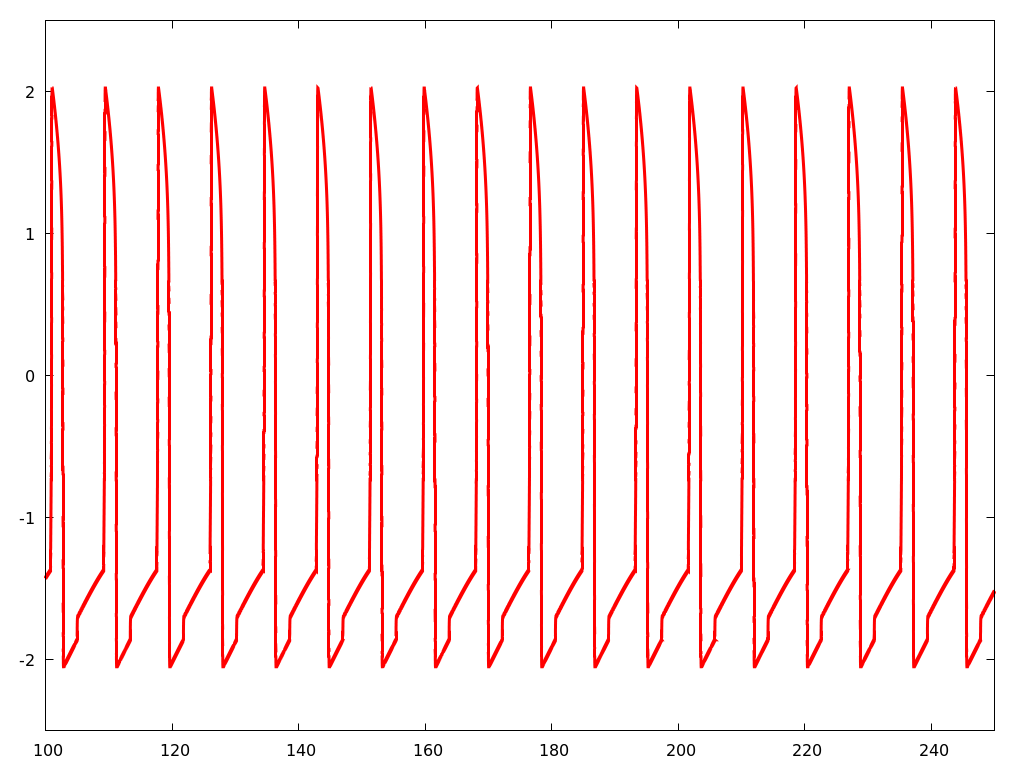}
\hfill
\includegraphics[width=3.5cm]{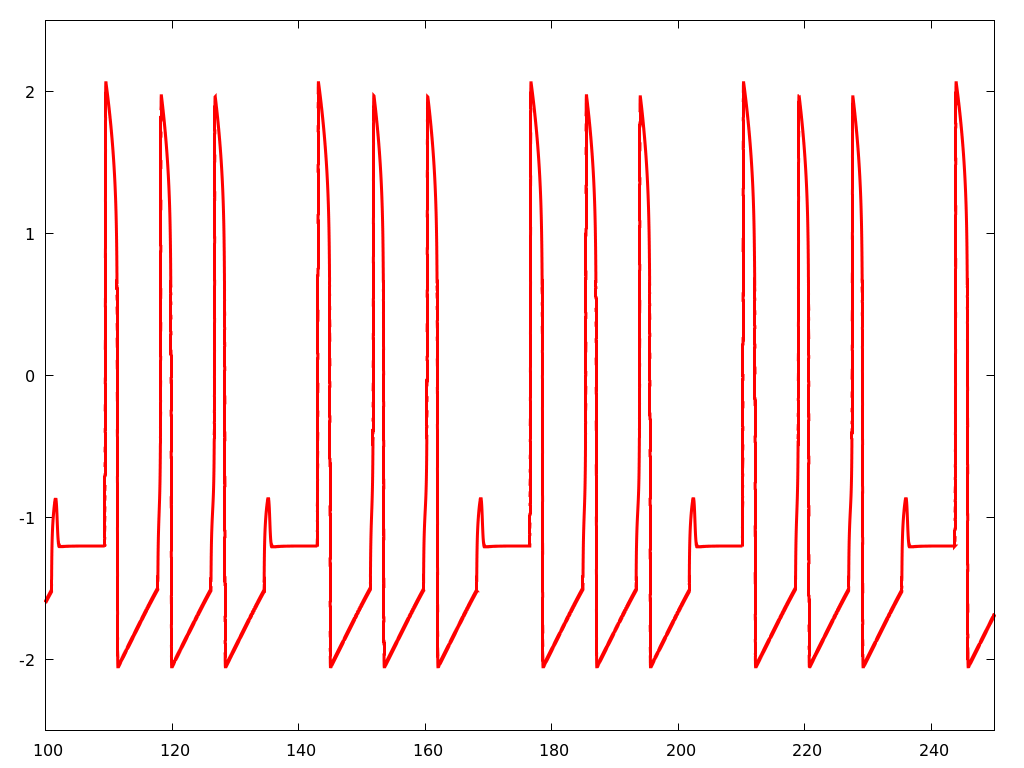}
\hfill
\includegraphics[width=3.5cm]{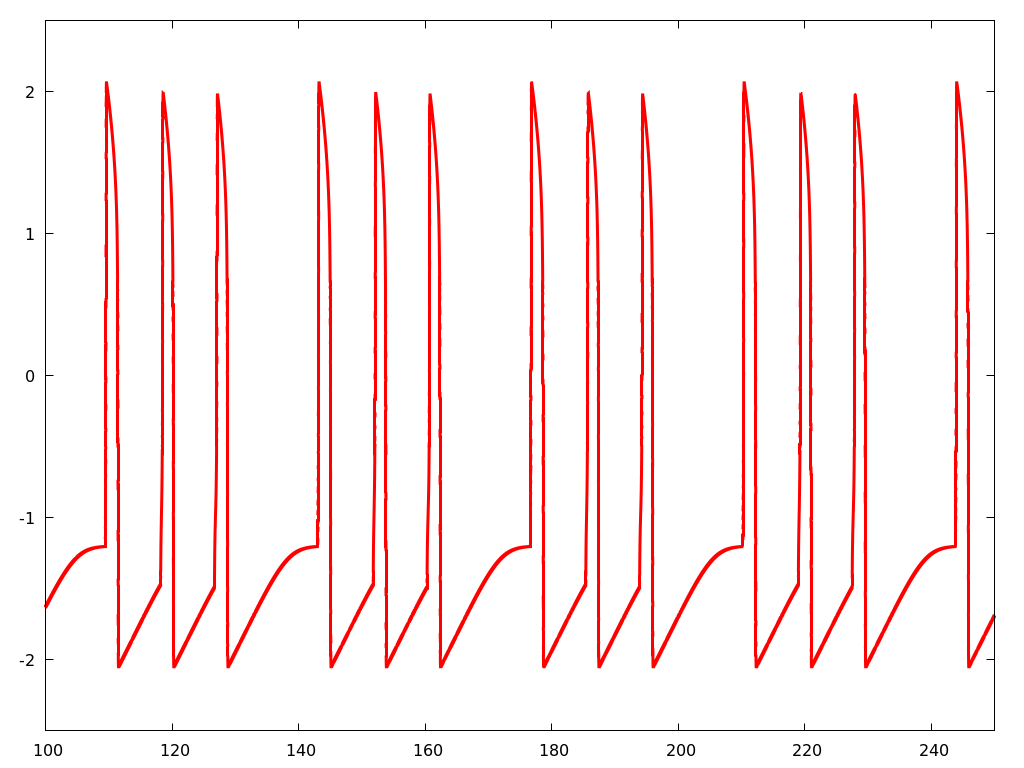}
\hfill
}
\caption{$\alpha=4.2$. 
\textbf{Left:} Voltage dynamics of $S_1$, simple periodic oscillation with period $2\alpha$ (twice the forcing due to the fact that every second stimulus arrives early, during the slow portion of the trajectory, so the dynamics continues toward the stable equilibrium after a brief rightward fast motion following that kick).
\textbf{Center:} Voltage dynamics of $S_2$, which receives periodic stimuli at period $2\alpha \in (\alpha_1,\alpha_0)$; this induces a complex MMO with three large-amplitude oscillations and one small-amplitude reset. The overall effective period is $8\alpha$. The important point of emphasis here is that while the stimulus this cell presents to $S_3$ is periodic, it is {\it not uniform}. (Effectively, there are three consecutive beats and one small flutter phase in each period.)
\textbf{Right:} Voltage dynamics of $S_3$ shows a simpler $8\alpha$-periodic oscillation with three depolarizations in each of its periods, interspersed by a quiescent phase distinct from the mixed-mode behavior in the second cell, in that the voltage ramps up instead of fluttering. This can be understood by noting the decomposition of the presented stimulus, which features three beats at $2\alpha$ together with a fourth flutter that does not produce a stimulus. So instead of making a small recovery (flutter loop), $S_3$ in fact has time to recover directly to a neighborhood of the equilibrium prior to the process renewing. For reference, we have omitted the plot for $S_4$ but it is identical to the one for $S_3$; this suggests that the situation settles as $j$ increases. Notice that the chain dynamics appears to converge to a TW that is periodic but qualitatively distinct from the simple regular TWs discussed in \cref{sec:simple_TW}. The latter TW features voltage dynamics at downstream sites that mix consecutive beats with single breaks in between, so they are not uniform. \label{fig:al=4p2cells1-2-3}
}
\end{center}
\end{figure}

\section{Conclusion / future work}
\label{sec:conclusion}
We have investigated the response of a FitzHugh-Nagumo cell to periodic forcing with instantaneous Dirac impulse, when the latter is incorporated into the recovery variable. Furthermore, we have provided detailed studies on the propagation of such a response in the context of a spatially extended feedforward network of cells communicating via excitatory kicks. The results of our parameter exploration (in decreasing forcing period) point to transitions from simple regular depolarization cascades to various mixed-mode oscillations and their propagation / filtration through the network. This work points to increasing complexity as forcing frequency is increased, in competition with a strong regularization or damping effect as the excitation progresses through the medium.

Of note, worthwhile future directions for this research include the investigation of direct stimulation to the voltage variable (for comparison with work done on \cref{E:benjamin-prior_work}), as well as more detailed exploration of the correspondence between stimulus size, period of forcing, and steady-state values in the parameter regime where the natural equilibrium is a sink of spiral type (which may afford the possibility for mixed-mode oscillations with higher number of small-amplitude voltage resets). Another interesting avenue of exploration would be to formulate a probabilistic analysis of the corresponding model under stochastic forcing with, e.g., a Poisson spike train. The interplay between the randomness of kick arrival times, modulo well defined average frequency, and the slow-fast geometry in the FHN dynamics poses an interesting mathematical challenge; it also brings the model's applicability to the realm of random stimuli, an altogether different area of applications.

Lastly, our studies have revealed a particularly interesting, if somewhat pure, mathematical challenge. The possible complexity of MMOs present in a single FHN cell driven by regular Dirac kicks is evident in the numerical investigations. Our admittedly heuristic explanation for this phenomenon relies on the presence of canard trajectories that veer toward the unstable branch of the critical manifold despite starting out as ``fast'' trajectories. Within this realm, one can attempt to address the question of whether it is possible for the steady state solution to feature an arbitrarily high number of long (threshold-crossing) canard trajectories prior to the small reset; that is, does this number increase without bound under adjustments to the period? or, is it bounded? Conversely, this question speaks to whether the length of time in between small-amplitude breaks can be arbitrarily large.



\bmhead{Acknowledgments}
The research of Benjamin Ambrosio (BA) on Neuroscience inspired modeling is partially funded by  CNRS (IEA00134). More generally BA wants to thank Le Havre Normandie University, LMAH, R\'egion Normandie, ISCN, the Courant Institute of Mathematical Sciences and New York University, the Hudson School of Mathematics, New Jersey, USA, for material and/or financial support  for research in Dynamical Systems, Complex Systems and its applications.

\section*{Declarations}


\begin{itemize}
\item[--] Funding: the research of BA has been partially funded by CNRS (IEA00134), LMAH, R\'egion Normandie  and the Hudson School of Mathematics.
\item[--] Conflict of interest/Competing interests. Authors declare no competing interests.
\item[--] Availability of data and materials: NA
\item[--] Both authors  contributed equally to this article. 
\end{itemize}
\bibliography{references}
\end{document}